\documentclass[11pt]{article}
\usepackage[latin9]{inputenc}
\usepackage{geometry}
\usepackage{color}
\usepackage{amsmath}
\usepackage{amsthm}
\usepackage{amssymb}
\usepackage[unicode=true,pdfusetitle,
 bookmarks=true,bookmarksnumbered=false,bookmarksopen=false,
 breaklinks=false,pdfborder={0 0 1},backref=page,colorlinks=true]
 {hyperref}
\hypersetup{
 linkcolor=blue}

\makeatletter
\numberwithin{equation}{section}
\numberwithin{figure}{section}
\theoremstyle{plain}
\newtheorem{thm}{\protect\theoremname}[section]
\theoremstyle{plain}
\newtheorem{lem}[thm]{\protect\lemmaname}
\theoremstyle{plain}
\newtheorem{cor}[thm]{\protect\corollaryname}
\newcommand{\lyxaddress}[1]{
	\par {\raggedright #1
	\vspace{1.4em}
	\noindent\par}
}

\@ifundefined{date}{}{\date{}}
\usepackage{graphicx}
\usepackage{appendix}

\usepackage{enumitem}
\usepackage[backref=page]{hyperref}

\makeatother

\providecommand{\corollaryname}{Corollary}
\providecommand{\lemmaname}{Lemma}
\providecommand{\theoremname}{Theorem}

\begin{document}
\global\long\def\F{\mathrm{\mathbf{F}} }%
\global\long\def\Aut{\mathrm{Aut}}%
\global\long\def\C{\mathbf{C}}%
\global\long\def\H{\mathbb{H}}%
\global\long\def\U{\mathbf{U}}%
\global\long\def\P{\mathcal{P}}%

\global\long\def\trace{\mathrm{tr}}%
\global\long\def\End{\mathrm{End}}%

\global\long\def\L{\mathcal{L}}%
\global\long\def\W{\mathcal{W}}%
\global\long\def\E{\mathbb{E}}%
\global\long\def\SL{\mathrm{SL}}%
\global\long\def\R{\mathbf{R}}%
\global\long\def\Pairs{\mathrm{PowerPairs}}%
\global\long\def\Z{\mathbf{Z}}%
\global\long\def\rs{\to}%
\global\long\def\A{\mathcal{A}}%
\global\long\def\a{\mathbf{a}}%
\global\long\def\rsa{\rightsquigarrow}%
\global\long\def\ep{\varepsilon}%
\global\long\def\b{\mathbf{b}}%
\global\long\def\df{\mathrm{def}}%
\global\long\def\eqdf{\stackrel{\df}{=}}%
\global\long\def\ZZ{\mathcal{Z}}%
\global\long\def\Tr{\mathrm{Tr}}%
\global\long\def\N{\mathbf{N}}%
\global\long\def\std{\mathrm{std}}%
\global\long\def\HS{\mathrm{H.S.}}%
\global\long\def\e{\mathbf{e}}%
\global\long\def\c{\mathbf{c}}%
\global\long\def\d{\mathbf{d}}%
\global\long\def\AA{\mathbf{A}}%
\global\long\def\BB{\mathbf{B}}%
\global\long\def\u{\mathbf{u}}%
\global\long\def\v{\mathbf{v}}%
\global\long\def\spec{\mathrm{spec}}%

\title{Effective lower bounds for spectra of random covers and random unitary
bundles}
\author{Will Hide}

\maketitle
\date{\vspace{-5ex}}
\begin{abstract}
{\footnotesize{}Let $X$ be a finite-area non-compact hyperbolic surface.
We study the spectrum of the Laplacian on random covering surfaces
of $X$ and on random unitary bundles over $X$. We show that there
is a constant $c>0$ such that, with probability tending to $1$ as
$n\to\infty$, a uniformly random degree-$n$ Riemannian covering
surface $X_{n}$ of $X$ has no Laplacian eigenvalues below $\frac{1}{4}-c\frac{\left(\log\log\log n\right)^{2}}{\log\log n}$
other than those of $X$ and with the same multiplicities. We also
show that with probability tending to $1$ as $n\to\infty$, a random
unitary bundle $E_{\phi}$ over $X$ of rank $n$ has no Laplacian
eigenvalues below $\frac{1}{4}-c\frac{\left(\log\log n\right)^{2}}{\log n}$.}{\footnotesize\par}
\end{abstract}
{\small{}\tableofcontents{}}{\small\par}

\section{Introduction}

Let $X$ be a finite-area non-compact hyperbolic surface, i.e. a Riemannian
surface with constant curvature $-1$. We study the Laplacian $\Delta_{X}$
on $L^{2}\left(X\right)$, whose spectrum $\spec\left(\Delta_{X}\right)$
is contained in $[0,\infty)$. The low-energy spectrum $\spec\left(\Delta_{X}\right)\cap[0,\frac{1}{4})$
consists of the trivial eigenvalue $0$ (which is simple if and only
if $X$ is connected) and possibly finitely many non-zero eigenvalues
of finite multiplicity. There is absolutely continuous spectrum in
$[\frac{1}{4},\infty)$ with possibly infinitely many eigenvalues
embedded in the continuous spectrum. Of particular interest to us
is the $\textit{spectral gap}$ $\inf\left(\spec\left(\Delta_{X}\right)\backslash\{0\}\right)$.

We study the size of the spectral gap for random surfaces. Our model,
studied first in \cite{MN20,MNP22}, is to fix a base surface $X$
and consider uniformly random covers $X_{n}$ of degree $n$. In this
context, since the $\spec\left(\Delta_{X}\right)\subset\spec\left(\Delta_{X_{n}}\right)$,
we restrict our attention to $\textit{new}$ eigenvalues, those appearing
in $\spec\left(\Delta_{X_{n}}\right)$ which do not appear in $\spec\left(\Delta_{X}\right)$.
Our covers will not need to be connected, but will be connected with
high probability by a result of Dixon \cite{Di69} (see Section \ref{sec:Set-up}).
We say that a family of events (depending on $n$) happens asymptotically
almost surely (a.a.s.) if they happen with probability tending to
$1$ as $n\to\infty$.

In this paper, we build upon the following theorem of Magee and the
author, from \cite{HM23}.
\begin{thm}[{\cite[Theorem 1.1]{HM23}}]
\label{HMMainTHM}Let $X$ be a finite-area non-compact hyperbolic
surface $X$, for any $\epsilon>0$, a uniformly random degree $n$
Riemannian cover $X_{n}$ a.a.s. satisfies 
\[
\spec\left(\Delta_{X_{n}}\right)\cap\left[0,\frac{1}{4}-\epsilon\right)=\spec\left(\Delta_{X}\right)\cap\left[0,\frac{1}{4}-\epsilon\right),
\]
where the multiplicities are the same on either side.
\end{thm}

The purpose of the current paper is to study which functions $\epsilon=\epsilon\left(n\right)$
one can take in Theorem \ref{HMMainTHM}. To this end, we show one
can take $\epsilon\to0$ at the rate $\frac{\left(\log\log\log n\right)^{2}}{\log\log n}$.
We prove the following.
\begin{thm}
\label{thm:Covers}For any finite-area non-compact hyperbolic surface
$X$, there exists a constant $c>0$ such that a uniformly random
degree $n$ Riemannian cover $X_{n}$ a.a.s. satisfies 
\[
\spec\left(\Delta_{X_{n}}\right)\cap\left[0,\frac{1}{4}-c\frac{\left(\log\log\log n\right)^{2}}{\log\log n}\right)=\spec\left(\Delta_{X}\right)\cap\left[0,\frac{1}{4}-c\frac{\left(\log\log\log n\right)^{2}}{\log\log n}\right),
\]
where the multiplicities are the same on either side.
\end{thm}

As a consequence of Theorem \ref{HMMainTHM}, it was shown in \cite[Section 8]{HM23}
that there exists a sequence of closed surfaces $\left\{ X_{i}\right\} $
with genera $g_{i}\to\infty$ with first non-zero eigenvalue $\lambda_{1}\left(X_{i}\right)\to\frac{1}{4}$,
resolving a conjecture of Buser \cite{Bu84} (see also \cite{LM22}
for an alternative proof). Theorem \ref{thm:Covers} allows the convergence
to $\frac{1}{4}$ to be made quantitative. Taking $X$ to be, for
example, the thrice punctured sphere which has $\lambda_{1}\left(X\right)>\frac{1}{4}$,
Theorem \ref{thm:Covers} produces a family of covers $X_{n}$ with
Euler characteristic $-n$ and
\[
\inf\spec\left(X_{n}\right)\geqslant\frac{1}{4}-c\frac{\left(\log\log\log n\right)^{2}}{\log\log n}.
\]
Taking covers of even degree, as explained in \cite[Section 8]{HM23},
one can then apply the compactification procedure of Buser, Burger
and Dodzuik \cite{BBD88} to get sequence of closed hyperbolic surfaces
$X_{g}$ with genus $g$ and 
\[
\lambda_{1}\left(X_{g}\right)\geqslant\frac{1}{4}-c'\frac{\left(\log\log\log g\right)^{2}}{\log\log g},
\]
giving a quantitative rate of converge to $\frac{1}{4}$ in the proof
of Buser's conjecture \cite{Bu84}. We refer the reader to \cite[Section 1.1]{HM23}
for the history of this conjecture. 

We are also interested in studying the analogous question for random
rank $n$ unitary bundles over $X$. The analogue of Theorem \ref{HMMainTHM}
in this context was proven by Zargar in \cite{Za22}, which we introduce
now. Since $X$ is a finite-area non-compact hyperbolic surface, $X$
can be realized as $X=\Gamma/\mathbb{H}$ where $\Gamma$ is a discrete
torsion free subgroup of $\text{PSL}_{2}\left(\mathbb{R}\right)$,
freely generated by $\gamma_{1},\dots,\gamma_{d}$ and any $\phi\in\text{Hom}\left(\Gamma,\text{U}\left(n\right)\right)$
is determined by $\phi\left(\gamma_{1}\right),\dots,\phi\left(\gamma_{d}\right)$.
We equip $\text{Hom}\left(\Gamma,\text{U}\left(n\right)\right)$ with
a natural probability measure $\mathbb{P}_{n}$ by sampling each $\phi\left(\gamma_{i}\right)\in\text{U}\left(n\right)$
independently with Haar probability. Let $\rho_{\phi}:\Gamma\to\text{U}\left(n\right)$
be the random $\mathbb{C}^{n}$ representation obtained via $\text{std}_{n}\circ\phi$
where $\text{std}_{n}$ is the standard representation. We consider
the associated (random) unitary bundle $E_{\phi}$ and the Laplacian
$\Delta_{\phi}$ on sections of $E_{\phi}$. Then $\spec\left(\Delta_{\phi}\right)\cap[0,\frac{1}{4})$
consists of finitely many eigenvalues with finite multiplicity. The
following was shown in \cite{Za22}.
\begin{thm}[{\cite[Theorem 1.2]{Za22}}]
\label{thm:Zargar} For any finite-area non-compact hyperbolic surface
$X$, for any $\kappa>0$, a random unitary bundle $E_{\phi}$ over
$X$ of rank $n$ has 
\[
\inf\spec\Delta_{\phi}\geqslant\frac{1}{4}-\kappa,
\]
a.a.s.
\end{thm}

We note that \cite[Theorem 1.2]{Za22} also deals with flat bundles
arising from other irreducible representations of $\text{U}\left(n\right)$,
subject to a condition on the signature. We prove the following quantitative
version of Theorem \ref{thm:Zargar}.
\begin{thm}
\label{Thm-Unitaries} For any finite-area non-compact hyperbolic
surface $X$, there exists a constant $c>0$ such that a random unitary
bundle $E_{\phi}$ over $X$ of rank $n$ has 
\[
\inf\spec\Delta_{\phi}\geqslant\frac{1}{4}-c\frac{\left(\log\log n\right)^{2}}{\log n},
\]
a.a.s. 
\end{thm}

\subsection{Other related works}

\subsubsection*{Random regular graphs}

Motivation for the results in this paper can be found in the setting
of random regular graphs. A celebrated theorem of Friedman \cite{Fr08},
formerly Alon's conjecture, says that for any $\epsilon>0$, a random
$d$-regular graph on $n$ vertices satisfies 
\begin{equation}
\lambda_{2},\left|\lambda_{n}\right|\leqslant2\sqrt{d-1}+\epsilon,\label{eq:friedman-theorem}
\end{equation}
with probability tending to $1$ as $n\to\infty$. It was shown by
Bordenave \cite{Bo20} that one can take $\epsilon$ in the above
to be $c\left(\frac{\log\log n}{\log n}\right)^{2}$. In an impressive
work of Huang and Yau \cite{HY21}, it was shown that one can take
$\epsilon=O\left(n^{-c}\right)$ for some $c>0$. 

It was conjectured by Friedman \cite{Fr03} that a version of Alon's
conjecture should hold for random covers of a fixed graph. This was
proved in a breakthrough work of Bordenave and Collins \cite{BC19}. 

\subsubsection*{Random covers}

The analogue of Theorem \ref{HMMainTHM} for Schottky surfaces was
proved by Magee and Naud in \cite{MN21} following an intermediate
result \cite{MN20}. Random covers of compact surfaces were studied
in \cite{MNP22} by Magee, Naud and Puder. They show that for any
$\epsilon>0$, (a.a.s.) a uniformly random degree $n$ cover has no
new eigenvalues below $\frac{3}{16}-\epsilon$. Eigenvalue statistics
for random covers have also been studied by Naud in \cite{Na22}.

\subsubsection*{Other models of random surfaces}

There has been much interest in studying the geometry and spectral
theory of random hyperbolic surfaces in various models. In \cite{BM04}
Brooks and Makover considered a combinatorial model of random surfaces,
showing the existence of a non-explicit uniform spectral gap with
high probability. Other works on the Brooks-Makover model include
\cite{Ga06,BCP21,SW22A}.

Another model of random surfaces is the Weil-Petersson model, arising
from sampling from moduli space with normalised Weil-Petersson volume.
Lengths of pants decompositions for compact surfaces in this model
were studied by Guth, Parlier and Young in \cite{GPY11}. Mirzakhani
\cite{Mi13} was the first to study the spectrum of the Laplacian
in this model proving that a random genus $g$ compact surface has
spectral gap of size at least $\frac{1}{4}\left(\frac{\log(2)}{2\pi+\log(2)}\right)^{2}\approx0.0024$
with probability tending to $1$ as $g\to\infty$. This was improved
to $\frac{3}{16}-\epsilon$ independently by Wu and Xue \cite{WX21}
and Lipnowksi and Wright \cite{LW21}, and subsequently to $\frac{2}{9}-\epsilon$
recently by Anantharaman and Monk \cite{AM23}.

For large volume non-compact surfaces, the number of cusps compared
to the genus has an impact on the low-energy spectrum. It was shown
by Zograf \cite{Zo87} that there exists a constant $C>0$ such that
any hyperbolic surface with genus $g$ and $n$ cusps has $\lambda_{1}<C\frac{g-1}{n}$.
If the number of cusps $n$ grows slower than $\sqrt{g}$ a Weil-Petersson
random surface has a uniform spectral gap with high probability \cite{Hi22,SW22}.
This fails to be true when the number of cusps grows faster than $\sqrt{g}$
\cite{SW22}, in this regime Weil-Petersson random surfaces have an
arbitrarily small eigenvalue. At the other extreme, if the genus is
fixed and the number of cusps $n$ tends to infinity, the number of
small eigenvalues is polynomial in $n$ with high probability \cite{HT22}.
Other works on spectral theory of Weil-Petersson random surfaces include
\cite{GMST21,Mo21,Ru22}.

\subsubsection*{Selberg's eigenvalue conjecture}

Spectral gaps for certain arithmetic hyperbolic surfaces are of great
interest in Number Theory, see e.g. \cite{Sa03}. Let $N\geqslant1$,
the principal congruence subgroup of $\text{SL}_{2}(\mathbb{Z})$
of level $N$ is 
\[
\Gamma\left(N\right)=\left\{ T\in\text{SL}_{2}(\mathbb{Z})\mid T\equiv I\mod N\right\} .
\]
Consider the quotient $X\left(N\right)\stackrel{\text{def}}{=}\Gamma\left(N\right)\backslash\mathbb{H}$.
Selberg proved in \cite{Se65} that 
\[
\lambda_{1}\left(X\left(N\right)\right)\geqslant\frac{3}{16},
\]
 for every $N\geqslant1$ and made the conjecture that $\frac{3}{16}$
can be improved to $\frac{1}{4}$. After many intermediate results
\cite{GJ78,Iw89,LRS95,Sa95,Iw96,KS02}, the best known lower bound
is $\frac{975}{4096}$ due to Kim and Sarnak \cite{Ki03}. In the
context of the current paper, taking $X$ to be a non-compact arithmetic
surface with $\lambda_{1}\left(X\right)\geqslant\frac{1}{4}$, Theorem
\ref{thm:Covers} shows that one can find a sequence of arithmetic
(not necessarily congruence) surfaces $\left\{ X_{n}\right\} _{n\in\mathbb{N}}$
with $\text{Vol}\left(X_{n}\right)=n\text{Vol}\left(X\right)$ and
\[
\lambda_{1}\left(X_{n}\right)\geqslant\frac{1}{4}-c\frac{\left(\log\log\log n\right)^{2}}{\log\log n}.
\]

\subsection{A word on the proof }

In \cite{HM23}, proof of Theorem \ref{HMMainTHM} relies on the work
of Bordenave and Collins in \cite{BC19}. The results of \cite{BC19}
were recently extended in a quantitative manor in \cite{BC23} which
is a crucial ingredient in the current paper. The method is similar
for unitary bundles so we restrict the discussion here to covering
surfaces.

Let $X=\Gamma\backslash\mathbb{H}$ be given. It is explained in Section
\ref{sec:Set-up} that one can parameterize degree $n$ covering surfaces
$X_{\varphi}$ by $\varphi\in\text{Hom}\left(\Gamma,S_{n}\right)$.
In \cite{HM23}, problem of forbidding new eigenvalues of a cover
$X_{\varphi}$ is reduced to bounding, with high probability, the
operator norm of a (random) operator of the form
\begin{equation}
\sum_{\gamma\in S}a_{\gamma}\otimes\rho_{\varphi}\left(\gamma\right),\label{eq:overvieweq1}
\end{equation}
where $a_{\gamma}\in M_{m}\left(\mathbb{C}\right)$, acting on $\mathbb{C}^{m}\otimes V_{n}^{0}$
where is the standard $n-1$ dimensional irreducible representation
of $S_{n}$. Here $m$ and $S\subset\Gamma$ are finite and fixed,
depending on the $\epsilon$ one chooses in Theorem \ref{HMMainTHM}.
The results of \cite{BC19} can be used to this end. 

In order to take $\epsilon(n)\to0$ as $n\to\infty$ as in Theorem
\ref{thm:Covers} one needs, amongst other things, to allow the size
of the set $S$ and the dimension $m$ of the matrices $a_{\gamma}$
to grow as a function of $n$, depending on $\epsilon$. The work
of Bordenave and Collins in \cite{BC23}, with an effective linearization
trick (Lemma \ref{lem:Lin3}), is able to deal with precisely this
situation. The proofs of Theorem \ref{thm:Covers} and Theorem \ref{Thm-Unitaries}
rely on effectivising the arguments in \cite{HM23} in order to apply
results from \cite{BC23}. 

We briefly discuss the rates obtained in Theorem \ref{thm:Covers}
and Theorem \ref{Thm-Unitaries}. Consider an operator $P$ the form
(\ref{eq:overvieweq1}). Let $l=\sup_{\gamma\in S}\text{wl}\left(\gamma\right)$
where $\text{wl}\left(\gamma\right)$ is the word-length in $\Gamma$
and $S\subset\Gamma$ be the index set. After applying a linearization
trick, in order to apply the results of Bordenave and Collins \cite{BC23}
to bound the operator norm of $P$, we eventually require 
\begin{equation}
l^{2}\left|S\right|^{\lceil\log_{2}l\rceil}l^{\left(\lceil\log_{2}l\rceil-1\right)}\ll\left(\log\left(n\right)\right)^{\frac{1}{4}},\label{eq:restraint 1}
\end{equation}
for permutations (c.f. Corollary \ref{thm:BC-Permutations-finalV})
and 
\begin{equation}
l^{2}\left|S\right|^{\lceil\log_{2}l\rceil}l^{\left(\lceil\log_{2}l\rceil-1\right)}\ll n^{\frac{1}{32d+160}},\label{eq:restraint2}
\end{equation}
for unitaries (c.f. Corollary \ref{thm:BC-unitaries-finalV}). It
is shown in Section \ref{sec:5} that in order to forbid spectra below
$\frac{1}{4}-\epsilon(n)$, one needs to consider an operator $P$
of the form (\ref{eq:overvieweq1}) where $\left|S\right|,l\ll\exp\left(\frac{2}{\epsilon(n)}\right)$.
This, together with (\ref{eq:restraint 1}) and (\ref{eq:restraint2})
governs the rates in Theorem \ref{thm:Covers} and Theorem \ref{Thm-Unitaries}. 

\subsection{Notation }

In the proofs, there will be quite a few constants, some of which
are important to track and some which are not. We choose to use the
notation $C>0$ throughout to denote some positive constant that only
depends (possibly) on the fixed surface $X$, whose precise value
is irrelevant. We warn that the value of $C$ sometimes changes from
line to line. Constants we do need to keep track of will be indexed
by a subscript in order of their first appearance, e.g. $c_{1}$,
or given as a numerical value.

For functions $f=f(n)$, $g=g(n)$ we use the Vinogradov notation
$f\ll g$ to mean that there exists a constant $c>0$ and an $N\in\mathbb{N}$
such that $f(n)\leqslant cg(n)$ for all $n\geqslant N$. We also
write $f=o(1)$ to mean that $f(n)\to0$ as $n\to\infty$.

\subsection*{Acknowledgments}

We thank Michael Magee and Charles Bordenave for discussions about
this work. 

\section{Set up\label{sec:Set-up}}

Let $X$ be a fixed non-compact finite-area hyperbolic surface. To
simplify notations, we assume $X$ has only one cusp. This will not
affect our arguments. We view $X$ as 
\[
X=\Gamma\backslash\mathbb{H},
\]
where $\Gamma$ is a discrete, torsion free subgroup of $\text{PSL}_{2}\left(\mathbb{R}\right)$,
freely generated by
\[
\gamma_{1},\ldots,\gamma_{d}\in\Gamma.
\]
We let $F$ be a Dirichlet domain for $X$. After possibly conjugating
$\Gamma$ in $\text{PSL}_{2}\left(\mathbb{R}\right)$, we can assume
that $\left(\begin{array}{cc}
1 & 1\\
0 & 1
\end{array}\right)\in\Gamma$, generating the stability group of the cusp and that $F\subset\text{\ensuremath{\left\{  x+iy\in\mathbb{H}\mid0\leqslant x\leqslant1\right\} } }$.
We then define
\begin{equation}
H\left(L\right)\eqdf\left\{ x+iy\in F\mid y>L\right\} .\label{eq:cusp-region}
\end{equation}

\subsubsection*{Random covers}

For any $n\in\mathbb{N}$, let $[n]\eqdf\{1,\ldots,n\}$ and $S_{n}$
denote the group of permutations of $[n]$. Given any $\varphi\in\text{Hom}\left(\Gamma,S_{n}\right)$
we define an action of $\Gamma$ on $\mathbb{H}\times[n]$ by 
\[
\gamma\left(z,x\right)\eqdf\left(\gamma z,\varphi(\gamma)[x]\right).
\]
We obtain a degree $n$ covering space $X_{\varphi}$ of $X$ by 
\begin{equation}
X_{\varphi}\eqdf\Gamma\backslash_{\varphi}\left(\mathbb{H}\times[n]\right).\label{eq:covering surface def}
\end{equation}
Sampling $\varphi$ uniformly randomly we obtain a uniformly random
degree $n$ cover. Note that $X_{\varphi}$ need not be connected,
indeed $X_{\varphi}$ is connected if and only if $\Gamma$ acts transitively
on $[n]$ via $\varphi$. By a Theorem of Dixon \cite{Di69}, two
uniformly random permutations in $S_{n}$ generate $A_{n}$ or $S_{n}$
a.a.s. hence a uniformly random cover $X_{\varphi}$ is connected
a.a.s. 

Let $V_{n}\eqdf\ell^{2}\left([n]\right)$ and $V_{n}^{0}\subset V_{n}$
the subspace of functions with zero mean. Then $S_{n}$ acts on $V_{n}$
via the standard representation $\text{std}_{n}$ by $0$-$1$ matrices,
and $V_{n}^{0}$ is an $n-1$ dimensional irreducible component. Given
a uniformly random $\varphi\in\text{Hom}\left(\Gamma;S_{n}\right)$,
we consider the random $V_{n}^{0}$ representation of $\Gamma$
\[
\rho_{\varphi}:\Gamma\to\text{End}\left(V_{n}^{0}\right),
\]
given by 
\[
\rho_{\varphi}\eqdf\text{std}_{n}\circ\varphi.
\]

\subsubsection*{Random unitary bundles }

Let $\text{U\ensuremath{\left(n\right)}}$ denote the unitary group.
Then a homomorphism $\phi:\Gamma\to\text{U}\left(n\right)$ is determined
uniquely by 
\[
\phi\left(\gamma_{1}\right),\dots,\phi\left(\gamma_{d}\right)\in\text{U}\left(n\right).
\]
We therefore can equip $\text{Hom}\left(\Gamma;\text{U}\left(n\right)\right)$
with a probability measure $\mathbb{P}_{n}$ by sampling the image
of each generator independently with Haar probability. Given such
a random $\phi\in\text{Hom}\left(\Gamma;\text{U}\left(n\right)\right)$,
we consider the random $\mathbb{C}^{n}$ representation of $\Gamma$
\[
\rho_{\phi}:\Gamma\to\text{U}\left(n\right),
\]
given by 
\[
\rho_{\phi}\eqdf\text{std}_{n}\circ\phi,
\]
where $\text{std}_{n}$ is the standard representation of $\text{U}\left(n\right)$.
Consider the action of $\Gamma$ on $\mathbb{H}\times\mathbb{C}^{n}$
by 
\[
\gamma\left(z,{\bf x}\right)\eqdf\left(\gamma z,\phi(\gamma){\bf x}\right),
\]
and let 
\[
E_{\phi}\eqdf\Gamma\backslash_{\phi}\left(\H\times\mathbb{C}^{n}\right)
\]
denote the quotient by this action. Then sampling $\text{\ensuremath{\phi\in}}\text{Hom}\left(\Gamma;\text{U}\left(n\right)\right)$
with probability $\mathbb{P}_{n}$, we obtain a random rank-$n$ unitary
bundle over $X$.

\subsection{Function spaces}

\subsubsection*{Covers}

For the convenience of the reader we recall the following function
spaces from \cite[Section 2.2]{HM23}. We define $L_{\text{new}}^{2}\left(X_{\varphi}\right)$
to be the space of $L^{2}$ functions on $X_{\varphi}$ orthogonal
to all lifts of $L^{2}$ functions from $X$. Then
\[
L^{2}\left(X_{\varphi}\right)\cong L_{\text{\text{new}}}^{2}\left(X\right)\oplus L^{2}\left(X\right).
\]
Recall that we fixed $F$ to be a Dirichlet fundamental domain for
$X$. Let $C^{\infty}\left(\mathbb{H};V_{n}^{0}\right)$ denote the
smooth $V_{n}^{0}$-valued functions on $\mathbb{H}$. There is an
isometric linear isomorphism between 
\[
C^{\infty}\left(X_{\varphi}\right)\cap L_{\text{\text{new}}}^{2}\left(X_{\varphi}\right),
\]
and the space of smooth $V_{n}^{0}$-valued functions on $\mathbb{H}$
satisfying 
\[
f\left(\gamma z\right)=\rho_{\varphi}\left(\gamma\right)f\left(z\right),
\]
 for all $\gamma\in\Gamma$, with finite norm
\[
\|f\|_{L^{2}(F)}^{2}\eqdf\int_{F}\|f(z)\|_{V_{n}^{0}}^{2}d\mu_{\mathbb{H}}\left(z\right)<\infty.
\]
We denote the space of such functions by $C_{\varphi}^{\infty}\left(\mathbb{H};V_{n}^{0}\right).$
The completion of $C_{\varphi}^{\infty}\left(\mathbb{H};V_{n}^{0}\right)$
with respect to $\|\bullet\|_{L^{2}(F)}$ is denoted by $L_{\varphi}^{2}\left(\H;V_{n}^{0}\right)$;
the isomorphism above extends to one between $L_{\text{\text{new}}}^{2}\left(X_{\varphi}\right)$
and $L_{\varphi}^{2}\left(\mathbb{H};V_{n}^{0}\right)$. 

Let $C_{c,\varphi}^{\infty}\left(\H;V_{n}^{0}\right)$ denote the
subset of $C_{\varphi}^{\infty}\left(\H;V_{n}^{0}\right)$ consisting
of functions which are compactly supported modulo $\Gamma$. We let
$H_{\varphi}^{2}\left(\H;V_{n}^{0}\right)$ denote the completion
of $C_{c,\varphi}^{\infty}\left(\H;V_{n}^{0}\right)$ with respect
to the norm
\[
\|f\|_{H_{\varphi}^{2}\left(\H;V_{n}^{0}\right)}^{2}\eqdf\|f\|_{L^{2}(F)}^{2}+\|\Delta f\|_{L^{2}(F)}^{2}.
\]
We let $H^{2}\left(X_{\varphi}\right)$ denote the completion of $C_{c}^{\infty}\left(X_{\varphi}\right)$
with respect to the norm 
\[
\|f\|_{H^{2}(X_{\varphi})}^{2}\eqdf\|f\|_{L^{2}(X_{\varphi})}^{2}+\|\Delta f\|_{L^{2}(X_{\varphi})}^{2}.
\]
Viewing $H^{2}\left(X_{\varphi}\right)$ as a subspace of $L^{2}\left(X_{\varphi}\right)$,
we let 
\[
H_{\text{new}}^{2}\left(X_{\varphi}\right)\eqdf H^{2}\left(X_{\varphi}\right)\cap L_{\text{new}}^{2}\left(X_{\varphi}\right).
\]
There is an isometric isomorphism between $H_{\text{new}}^{2}\left(X_{\varphi}\right)$
and $H_{\varphi}^{2}\left(\H;V_{n}^{0}\right)$ that intertwines the
two relevant Laplacian operators.

\subsubsection*{Unitary bundles}

We identify the space of smooth global sections of $E_{\phi}\to X$,
denoted \textbf{$C^{\infty}\left(X;E_{\phi}\right)$}, with the space
$C_{\phi}^{\infty}\left(\H;\mathbb{C}^{n}\right)$ of smooth $\mathbb{C}^{n}$
valued functions on $\mathbb{H}$ which transform as 
\[
f\left(\gamma z\right)=\rho_{\phi}\left(\gamma\right)f\left(z\right).
\]
We let $L^{2}\left(X;E_{\phi}\right)$ denote the completion of $C^{\infty}\left(X;E_{\phi}\right)$
with respect to the norm
\[
\|f\|_{L^{2}\left(F\right)}^{2}\eqdf\int_{F}\|f(z)\|_{\mathbb{C}^{n}}^{2}d\mu(z).
\]
Analogously to the case of covers, we define $C_{c,\phi}^{\infty}\left(\H;\mathbb{C}^{n}\right)$
to be the subset of $C_{\phi}^{\infty}\left(\H;\mathbb{C}^{n}\right)$
of functions which are compactly supported modulo $\Gamma$ and $H^{2}\left(X;E_{\phi}\right)$
to be the completion of $C_{c,\phi}^{\infty}\left(\H;\mathbb{C}^{n}\right)$
with respect to the norm
\[
\|f\|_{H_{\phi}^{2}\left(\H;\mathbb{C}^{n}\right)}^{2}\eqdf\|f\|_{L^{2}(F)}^{2}+\|\Delta f\|_{L^{2}(F)}^{2}.
\]

\section{\label{sec:Random-matrix-theory}Random matrix theory }

In this section we introduce the necessary random matrix theory results.
Recall $\Gamma$ is a free group on $d$ generators $\gamma_{1},\dots,\gamma_{d}$.
The wordlength $\text{wl}\left(\gamma\right)$ is the length of $\gamma$
as a reduced word in $\gamma_{1},\dots,\gamma_{d},\gamma_{1}^{-1},\dots,\gamma_{d}^{-1}$.
Let $\lambda:\Gamma\to\text{End}\left(l^{2}\left(\Gamma\right)\right)$
denote the right regular representation of $\Gamma$. We rely heavily
on recent extremely powerful results of Bordenave and Collins \cite{BC23}.
Firstly, for random unitaries, we need the following.
\begin{thm}[{\cite[Corollary 1.2]{BC23} and \cite[Lemma 7.4]{BC23} }]
\label{BCUnitarieslin}Let $m\leqslant n^{\frac{1}{32d+160}}$and
$a_{0},a_{1},\ldots a_{d}\in M_{m}\left(\mathbb{C}\right)$ with $a_{0}=a_{0}^{*}$.
Then there exists a constant $c_{1}>0$ such that for a random $\phi\in\left(\textup{Hom}\left(\Gamma,U\left(n\right)\right),\mathbb{P}_{n}\right)$,
with probability at least $1-\exp\left(-\sqrt{n}\right)$,
\begin{align*}
 & \left\Vert a_{0}\otimes\mathrm{Id}_{\mathbb{C}^{n}}+\sum_{i=1}^{d}\left(a_{i}\otimes\rho_{\phi}\left(\gamma_{i}\right)+a_{i}^{*}\otimes\rho_{\phi}\left(\gamma_{i}^{-1}\right)\right)\right\Vert _{\mathbb{C}^{m}\otimes\mathbb{C}^{n}}\\
\leqslant & \left\Vert a_{0}\otimes\mathrm{Id}_{\ell^{2}(\Gamma)}+\sum_{i=1}^{d}\left(a_{i}\otimes\lambda\left(\gamma_{i}\right)+a_{i}^{*}\otimes\lambda\left(\gamma_{i}^{-1}\right)\right)\right\Vert _{\mathbb{C}^{m}\otimes\ell^{2}(\Gamma)}\left(1+\frac{c_{1}}{n^{\frac{1}{32d+160}}}\right).
\end{align*}
\end{thm}

We also need the following result for random permutations.
\begin{thm}[{\cite[Corollary 1.4]{BC23}}]
\label{BCPermslin}Let $m\leqslant n^{\sqrt{\log n}}$ and $a_{0},a_{1},\ldots a_{d}\in M_{m}\left(\mathbb{C}\right)$
with $a_{0}=a_{0}^{*}$. Then there exists a constant $c_{2}>0$ such
that for a uniformly random $\varphi\in\textup{Hom}\left(\Gamma,S_{n}\right)$,
with probability at least $1-\frac{c_{2}}{\sqrt{n}}$, 
\begin{align*}
 & \left\Vert a_{0}\otimes\mathrm{Id}_{V_{n}^{0}}+\sum_{i=1}^{d}\left(a_{i}\otimes\rho_{\varphi}\left(\gamma_{i}\right)+a_{i}^{*}\otimes\rho_{\varphi}\left(\gamma_{i}^{-1}\right)\right)\right\Vert _{\mathbb{C}^{m}\otimes V_{n}^{0}}\\
\leqslant & \left\Vert a_{0}\otimes\mathrm{Id}_{\ell^{2}(\Gamma)}+\sum_{i=1}^{d}\left(a_{i}\otimes\lambda\left(\gamma_{i}\right)+a_{i}^{*}\otimes\lambda\left(\gamma_{i}^{-1}\right)\right)\right\Vert _{\mathbb{C}^{m}\otimes\ell^{2}(\Gamma)}\left(1+\frac{c_{2}}{\left(\log n\right)^{\frac{1}{4}}}\right).
\end{align*}
\end{thm}

Theorem \ref{BCUnitarieslin} and Theorem \ref{BCPermslin} both concern
linear polynomials. The polynomial to which we shall want to apply
Theorem \ref{BCUnitarieslin} and Theorem \ref{BCPermslin} will not
be linear so we need to apply a linearization procedure in order to
access these bounds. The idea is that we can trade a polynomial of
large degree for one of smaller degree at a cost of replacing $M_{m}\left(\mathbb{C}\right)$
by $M_{m}\left(\mathbb{C}\right)\otimes M_{k}\left(\mathbb{C}\right)$
for some $k$. This procedure is known as the linearization trick
\cite{Pi96}, \cite{HT05}. 

We use an effective linearization proved in \cite[Section 8]{BC23}.
In \cite[Section 8]{BC23}, the authors considered operators of the
form $\sum_{g\in B_{l}}a_{\gamma}\otimes\rho_{\phi}\left(\gamma\right)$
where $B_{l}$ is the ball of size $l$ in the word metric of $\Gamma$
with our fixed choice of generators (note that linear means $l\leqslant1$
in this context). In our case, the operators we want to consider will
be of the form $\sum_{g\in S}a_{\gamma}\otimes\rho_{\phi}\left(\gamma\right)$
where $S\subset B_{l}$ where $\left|S\right|$ is roughly of size
$l$ which shall give us a quantitative saving. This is only a minor
adaptation to the arguments in \cite[Section 8]{BC23}, however since
this is a key point for our method, we include the details. 

We say that a subset $S\subset\Gamma$ is symmetric if $g\in S$ implies
$g^{-1}\in S$.
\begin{lem}
\label{lem:in1}Let $l\geqslant2$ be an even integer and let $S\subset B_{l}$.
Consider $\left(a_{g}\right)_{g\in S}$ with $a_{g}\in M_{m}\left(\mathbb{C}\right)$.
Then there exists a symmetric set $S_{1}\subset B_{\frac{l}{2}}$
with $\left|S_{1}\right|\leqslant4\left|S\right|$, $\left(b_{g}\right)_{g\in S_{1}}$
with $b_{g}\in M_{m}\left(\mathbb{C}\right)\otimes M_{2\left|S_{1}\right|}\left(\mathbb{C}\right)$
and $\theta\geqslant0$ such that for any unitary representation $\left(\rho,V\right)$
of $\Gamma$,
\[
\|\sum_{\gamma\in S}a_{\gamma}\otimes\rho\left(\gamma\right)\|_{\mathbb{C}^{m}\otimes V}=\|\sum_{\gamma\in S_{1}}b_{\gamma}\otimes\rho\left(\gamma\right)\|_{\mathbb{C}^{m}\otimes\mathbb{C}^{2\left|S_{1}\right|}\otimes V}^{2}-\theta,
\]
where 
\[
\theta\leqslant4\left|S\right|\|\sum_{\gamma\in S}a_{\gamma}\otimes\lambda\left(\gamma\right)\|_{\mathbb{C}^{m}\otimes l^{2}\left(\Gamma\right)}.
\]
\end{lem}

\begin{proof}
We consider a set $S_{1}\subset B_{\frac{l}{2}}$ such that 
\[
S\subset\left\{ g^{-1}h\mid g,h\in S_{1}\right\} .
\]
We claim we can choose $S_{1}$ so that 
\[
\left|S_{1}\right|\leqslant4\left|S\right|.
\]
Indeed if $w\in S\cap B_{\frac{l}{2}}$, we can just take $w$ and
the identity to be in $S_{1}$. If $w\in S$ has word-length $>\frac{l}{2}$,
then it can be written as $g^{-1}h$ for two words $g,h\in B_{\frac{l}{2}}$
and we add both words to $S_{1}$. We make $S_{1}$ symmetric by including
the inverses of any word already added, at worst doubling the size
of $S_{1}$. 

Note that we can enlarge $S$ to a symmetric set without changing
the size of $S_{1}$, since $S_{1}$ is symmetric. After possibly
replacing $M_{m}\left(\mathbb{C}\right)$ with $M_{m}\left(\mathbb{C}\right)\otimes M_{2}\left(\mathbb{C}\right)$
and enlarging $S$ to a symmetric set, we can assume that the symmetry
condition $a_{\gamma}=a_{\gamma^{-1}}^{*}$ holds, in particular $P\eqdf\sum_{\gamma\in S}a_{\gamma}\otimes\rho\left(\gamma\right)$
is self-adjoint e.g. \cite[Proof of Theorem 1.1]{BC23}. 

We now follow \cite[Proof of Lemma 8.1]{BC23}. Consider the element
$\tilde{a}\in M_{m}\left(\mathbb{C}\right)\otimes M_{\left|S_{1}\right|}\left(\mathbb{C}\right)$
defined by $\left(\tilde{a}_{g,h}\right)_{g,h\in S_{1}}$, 
\[
\tilde{a}_{g,h}=\frac{1}{\#\left\{ \left(g',h'\right)\in S_{1}\times S_{1}\mid\left(g'\right)^{-1}h'=g^{-1}h\right\} }a_{g^{-1}h},
\]
 when $g^{-1}h\in S$ and $\tilde{a}_{g,h}=0$ otherwise. Then 
\[
\sum_{\substack{g,h\in S_{1}\\
g^{-1}h=w\in S
}
}\tilde{a}_{g,h}=a_{w}.
\]
We have
\[
\|\tilde{a}\|^{2}\leqslant\|\sum_{g,h\in S_{1}}\tilde{a}_{g,h}\tilde{a}_{g,h}^{*}\|\leqslant\|\sum_{w\in S}a_{w}a_{w}^{*}\|\leqslant\|\sum_{w\in S}a_{w}\otimes\lambda\left(w\right)\|^{2}.
\]
The operator $\tilde{a}+\|\tilde{a}\|\text{Id}_{m\left|S_{1}\right|}$
is positive semi-definite and we let $\tilde{b}\in M_{m}\left(\mathbb{C}\right)\otimes M_{\left|S_{1}\right|}\left(\mathbb{C}\right)$
be its self-adjoint square root. For $g\in S_{1}$ we define 
\[
b_{g}\eqdf\tilde{b}\left(\text{Id}_{m}\otimes e_{g,\emptyset}\right)\in M_{m}\left(\mathbb{C}\right)\otimes M_{\left|S_{1}\right|}\left(\mathbb{C}\right),
\]
where $e_{g,h}\eqdf\delta_{g}\otimes\delta_{h}\in M_{B_{\frac{l}{2}}}\left(\mathbb{C}\right)$
and $\emptyset$ is the unit in $\Gamma$. Then defining 
\[
Q\eqdf\sum_{g\in S_{1}}b_{g}\otimes\rho(g),
\]
we have 
\begin{align*}
Q^{*}Q & =\sum_{g,h\in S_{1}}\left(\text{Id}_{m}\otimes e_{\emptyset,g}\right)\tilde{b}^{2}\left(\text{Id}_{m}\otimes e_{h,\emptyset}\right)\otimes\rho(g^{-1}h)\\
 & =\sum_{g,h\in S_{1}}e_{\emptyset,\emptyset}\otimes\left(\tilde{a}_{g,h}+\|\tilde{a}\|{\bf 1}_{g=h}\text{Id}_{m}\right)\otimes\rho(g^{-1}h)\\
 & =e_{\emptyset,\emptyset}\otimes\left(\sum_{\gamma\in S}a_{\gamma}\otimes\rho\left(\gamma\right)+\theta\text{Id}_{\mathbb{C}^{m}\otimes V}\right),
\end{align*}
where 
\[
\theta\leqslant\left|S_{1}\right|\|\sum_{\gamma\in S}a_{\gamma}\otimes\lambda\left(\gamma\right)\|\leqslant4\left|S\right|\|\sum_{\gamma\in S}a_{\gamma}\otimes\lambda\left(\gamma\right)\|.
\]
It follows that 
\[
\|Q\|^{2}=||\sum_{\gamma\in S}a_{\gamma}\otimes\rho\left(\gamma\right)+\theta\text{Id}_{\mathbb{C}^{m}\otimes V}\|_{\mathbb{C}^{m}\otimes V}=\|\sum_{\gamma\in S}a_{\gamma}\otimes\rho\left(\gamma\right)\|_{\mathbb{C}^{m}\otimes V}+\theta.
\]
\end{proof}
We can iterate this process to obtain the following, c.f. \cite[Lemma 8.2]{BC23}. 
\begin{lem}
\label{lem:lin2}Let $l\geqslant2$ be an integer, $S\subset B_{l}$
and let $v=\lceil\log_{2}l\rceil$. Then for each $k\in\left\{ 0,\dots,v\right\} $
there is:
\begin{itemize}
\item An integer $n_{k}\geqslant1$ with $n_{v}\leqslant2l\left|S\right|^{\lceil\log_{2}l\rceil}l^{\left(\lceil\log_{2}l\rceil-1\right)}.$
\item A symmetric set $S_{k}\subset B_{2^{v-k}}$ with $S_{0}=S$, $\left|S_{k}\right|\leqslant\min\left\{ 4^{k}\left|S\right|,\left|B_{2^{v-k}}\right|\right\} $
. 
\item A set $\left(a_{g}^{k}\right)_{g\in S_{k}}$ with $a_{g}^{k}\in M_{m}\left(\mathbb{C}\right)\otimes M_{n_{k}}\left(\mathbb{C}\right)$. 
\item A constant $\theta_{k}\geqslant0$ such that for $k\geqslant1$, 
\[
\theta_{k}\leqslant\big\|\sum_{\gamma\in S_{k-1}}a_{\gamma}^{k-1}\otimes\lambda\left(\gamma\right)\big\|_{\mathbb{C}^{m}\otimes\mathbb{C}^{n_{k-1}}\otimes l^{2}\left(\Gamma\right)}\left|S_{k}\right|,
\]
\end{itemize}
such that for any unitary representation $\left(\rho,V\right)$ of
$\Gamma$,

\[
\|\sum_{\gamma\in S_{k-1}}a_{g}^{k-1}\otimes\rho(\gamma)\|_{\mathbb{C}^{m}\otimes\mathbb{C}^{n_{k-1}}\otimes V}=\|\sum_{\gamma\in S_{k}}a_{g}^{k}\otimes\rho(\gamma)\|_{\mathbb{C}^{m}\otimes\mathbb{C}^{n_{k}}\otimes V}^{2}-\theta_{k}.
\]

\end{lem}

\begin{proof}
This is a straightforward consequence of iterating the procedure of
Lemma \ref{lem:in1}. We have 
\[
n_{v}\leqslant\prod_{i=1}^{v}2\left|S_{i}\right|\leqslant\prod_{i=1}^{v}2\cdot4^{i}\cdot\left|S\right|=2^{v}4^{\frac{v(v-1)}{2}}\left|S\right|^{v},
\]
where $v=\lceil\log_{2}l\rceil$ which gives
\[
n_{v}\leqslant2l\left|S\right|^{\lceil\log_{2}l\rceil}l^{\left(\lceil\log_{2}l\rceil-1\right)}.
\]
\end{proof}
As a consequence we obtain the following, c.f. \cite[Lemma 8.3]{BC23}
\begin{lem}
\label{lem:Lin3}Let $l\geqslant2$ be an integer, $S\subset B_{l}$
and set $v=\lceil\log_{2}l\rceil$. Consider $\left(a_{g}^{v}\right)_{g\in S_{v}}$
as in Lemma \ref{lem:lin2} and denote $a_{0}=a_{\emptyset}^{v}$,
$a_{i}=a_{\gamma_{i}}^{v}$ for $1\leqslant i\leqslant2d$. Let $\left(\rho,V\right)$
be any unitary representation of $\Gamma$. We have that for $0<\epsilon<1$,
if 
\[
2\epsilon l^{2}\left|S\right|^{\lceil\log_{2}l\rceil}l^{\left(\lceil\log_{2}l\rceil-1\right)}<1
\]
 and 
\[
\|a_{0}\otimes\textup{Id}_{V}+\sum_{i=1}^{2d}a_{i}\otimes\rho(\gamma_{i})\|_{\mathbb{C}^{m}\otimes\mathbb{C}^{n_{v}}\otimes V}\leqslant\|a_{0}\otimes\textup{Id}_{l^{2}\left(\Gamma\right)}+\sum_{i=1}^{2d}a_{i}\otimes\lambda(\gamma_{i})\|_{\mathbb{C}^{m}\otimes\mathbb{C}^{n_{v}}\otimes l^{2}\left(\Gamma\right)}\left(1+\epsilon\right),
\]
then 
\[
\|\sum_{\gamma\in S}a_{\gamma}\otimes\rho\left(\gamma\right)\|_{\mathbb{C}^{m}\otimes V}\leqslant\|\sum_{\gamma\in S}a_{\gamma}\otimes\lambda\left(\gamma\right)\|_{\mathbb{C}^{m}\otimes l^{2}\left(\Gamma\right)}\left(1+2\epsilon l^{2}\left|S\right|^{\lceil\log_{2}l\rceil}l^{\left(\lceil\log_{2}l\rceil-1\right)}\right).
\]
\end{lem}

\begin{proof}
For $k\in\left\{ 1,\dots,v\right\} $, let $a_{g}^{k}\in M_{m}\left(\mathbb{C}\right)\otimes M_{n_{k}}\left(\mathbb{C}\right)$
for $g\in S_{k}$ be as given by Lemma \ref{lem:lin2}. For some $k\in\left\{ 1,\dots,v\right\} $,
assume that for some $0<\epsilon_{k}<1,$
\[
\|\sum_{\gamma\in S_{k}}a_{\gamma}^{k}\otimes\rho\left(\gamma\right)\|\leqslant\|\sum_{\gamma\in S_{k}}a_{\gamma}^{k}\otimes\lambda\left(\gamma\right)\|\left(1+\epsilon_{k}\right).
\]
 Then by Lemma \ref{lem:lin2} applied twice,
\begin{align}
\|\sum_{\gamma\in S_{k-1}}a_{\gamma}^{k-1}\otimes\rho\left(\gamma\right)\|-\|\sum_{\gamma\in S_{k-1}}a_{\gamma}^{k-1}\otimes\lambda\left(\gamma\right)\| & =\|\sum_{\gamma\in S_{k}}a_{\gamma}^{k}\otimes\rho\left(\gamma\right)\|^{2}-\|\sum_{\gamma\in S_{k}}a_{\gamma}^{k}\otimes\lambda\left(\gamma\right)\|^{2}\nonumber \\
 & \leqslant\epsilon_{k}\left(1+2\epsilon_{k}\right)\left(\|\sum_{\gamma\in S_{k-1}}a_{\gamma}^{k-1}\otimes\lambda\left(\gamma\right)\|+\theta_{k}\right)\nonumber \\
 & \leqslant4\cdot4^{k}\left|S\right|\epsilon_{k}\|\sum_{\gamma\in S_{k-1}}a_{\gamma}^{k-1}\otimes\lambda\left(\gamma\right)\|.\label{eq:inequality-lins}
\end{align}
By assumption, $\epsilon_{v}=\epsilon<1$ and then by setting $\epsilon_{k-1}\eqdf4\cdot4^{k}\left|S\right|\epsilon_{k}$
(recalling $\theta_{k}\leqslant4^{k}\left|S\right|\|\sum_{\gamma\in S_{k-1}}a_{\gamma}^{k-1}\otimes\lambda\left(\gamma\right)\|$
from Lemma \ref{lem:lin2}), By the definition of $\epsilon_{k-j}$
we see
\[
\epsilon_{0}=\epsilon\prod_{i=1}^{v}4\cdot4^{i}\left|S\right|\leqslant2\epsilon l^{2}\left|S\right|^{\lceil\log_{2}l\rceil}l^{\left(\lceil\log_{2}l\rceil-1\right)},
\]
If one picks $2\epsilon l^{2}\left|S\right|^{\lceil\log_{2}l\rceil}l^{\left(\lceil\log_{2}l\rceil-1\right)}<1$
then this ensures that $\epsilon_{k-j}<1$ for $j=1,\dots,k$ and
we can apply the inequality (\ref{eq:inequality-lins}) inductively
starting from $k=v$ to $k=1$ provided that each subsequent $\epsilon_{k-j}<1$.
\end{proof}
By Lemma \ref{lem:Lin3}, we obtain the following corollary from Theorem
\ref{BCUnitarieslin}.
\begin{cor}
\label{thm:BC-unitaries-finalV}Let $m\geqslant1$, $l\geqslant2$
and let $S\subset B_{l}$ be a finite set such that 
\[
2ml\left|S\right|^{\lceil\log_{2}l\rceil}l^{\left(\lceil\log_{2}l\rceil-1\right)}\leqslant\exp\left(n^{\frac{1}{32d+160}}\right),
\]
and

\[
2c_{1}l^{2}\left|S\right|^{\lceil\log_{2}l\rceil}l^{\left(\lceil\log_{2}l\rceil-1\right)}\leqslant n^{\frac{1}{32d+160}},
\]
where $c_{1}$ is the constant in Theorem \ref{thm:BC-unitaries-finalV}.
Let $\gamma\mapsto a_{\gamma}\in M_{m}\left(\mathbb{C}\right)$ be
any map supported in $S$. For a random $\phi\in\left(\textup{Hom}\left(\Gamma,U\left(n\right)\right),\mathbb{P}_{n}\right),$with
probability at least $1-\exp\left(-\sqrt{n}\right)$ one has 
\[
\|\sum_{\gamma\in S}a_{\gamma}\otimes\rho_{\phi}\left(\gamma\right)\|_{\mathbb{C}^{m}\otimes\mathbb{C}^{n}}\leqslant\|\sum_{\gamma\in S}a_{\gamma}\otimes\lambda\left(\gamma\right)\|_{\mathbb{C}^{m}\otimes l^{2}\left(\Gamma\right)}\left(1+c_{1}\frac{2l^{2}\left|S\right|^{\lceil\log_{2}l\rceil}l^{\left(\lceil\log_{2}l\rceil-1\right)}}{n^{\frac{1}{32d+160}}}\right).
\]
\end{cor}

By applying Lemma \ref{lem:Lin3} and Theorem \ref{BCPermslin} have
an analogous corollary for permutation matrices. 
\begin{cor}
\label{thm:BC-Permutations-finalV}Let $m$ and $l$ satisfy 
\[
2ml\left|S\right|^{\lceil\log_{2}l\rceil}l^{\left(\lceil\log_{2}l\rceil-1\right)}\leqslant n^{\sqrt{\log n}}.
\]
Let $S\subset B_{l}$ be a finite set whose size satisfies 
\[
2c_{2}l^{2}\left|S\right|^{\lceil\log_{2}l\rceil}l^{\left(\lceil\log_{2}l\rceil-1\right)}\leqslant\left(\log\left(n\right)\right)^{\frac{1}{4}},
\]
where $c_{2}$ is the constant in Theorem \ref{thm:BC-Permutations-finalV}.
Let $\gamma\mapsto a_{\gamma}\in M_{m}\left(\mathbb{C}\right)$ be
any map supported in $S$. For a uniformly random $\varphi\in\textup{Hom}\left(\Gamma,S_{n}\right)$,
with probability at least $1-\frac{c_{2}}{\sqrt{n}}$ one has 
\[
\|\sum_{\gamma\in S}a_{\gamma}\otimes\rho_{\varphi}\left(\gamma\right)\|_{\mathbb{C}^{m}\otimes V_{n}^{0}}\leqslant\|\sum_{\gamma\in S}a_{\gamma}\otimes\lambda\left(\gamma\right)\|_{\mathbb{C}^{m}\otimes l^{2}\left(\Gamma\right)}\left(1+c_{2}\frac{2l^{2}\left|S\right|^{\lceil\log_{2}l\rceil}l^{\left(\lceil\log_{2}l\rceil-1\right)}}{\left(\log\left(n\right)\right)^{\frac{1}{4}}}\right).
\]
\end{cor}

\section{Construction of the parametrix}

Our parametrix construction is the same as in \cite{HM23}. 

\subsection{\label{subsec:Cusp-parametrix}Cusp parametrix}

In this subsection we introduce the cuspidal part of the parametrix. 

Recall that we made the assumption that $X$ has only one cusp to
simplify notation. We identify the cusp $\mathcal{C}$ with 
\[
\mathcal{C}\eqdf\left(1,\infty\right)\times S^{1}
\]
with the metric 
\[
\frac{dr^{2}+dx^{2}}{r^{2}},
\]
where $\left(r,x\right)\in\left(1,\infty\right)\times S^{1}$. For
each $n\in\mathbb{N}$ we will define the cutoff functions $\chi_{\mathcal{C},n}^{+},\chi_{\mathcal{C},n}^{-}:\mathcal{C}\to[0,1]$
to be functions that are identically zero in a neighborhood of $\{1\}\times S^{1}$,
identically equal to $1$ in a neighborhood of $\{\infty\}\times S^{1}$,
such that 
\begin{equation}
\chi_{\mathcal{C},n}^{+}\chi_{\mathcal{C},n}^{-}=\chi_{\mathcal{C},n}^{-}.\label{eq:stagger}
\end{equation}
We extend $\chi_{\mathcal{C},n}^{\pm}$ by $0$ to functions on $X$.
Let $\kappa:\mathbb{N}\to(0,\infty)$ be some given function. Later
on (Lemma \ref{lem:Radnomopbound}) we shall pick specific functions
$\kappa\left(n\right)$ for the case of covers and unitary bundles.
As indicated by the subscript, the functions $\chi_{\mathcal{C},n}^{+},\chi_{\mathcal{C},n}^{-}$
will depend on $n$ through the function $\kappa(n)$.
\begin{lem}
\label{Bounding-cutoff-derivative}Given $\kappa:\mathbb{N}\to(0,\infty)$,
for each $n\in\mathbb{N}$ we can choose $\chi_{\mathcal{C},n}^{\pm}$
as above so that 
\[
\|\nabla\chi_{\mathcal{C},n}^{+}\|_{\infty},\|\Delta\chi_{\mathcal{C},n}^{+}\|_{\infty}\leq\frac{\kappa\left(n\right)}{30}.
\]
\end{lem}

\begin{proof}
One can find a $\tau_{0}>1$ and a smooth function $\chi_{\mathcal{C},0}^{+}:[0,\infty)\to[0,1]$
with $\chi_{\mathcal{C},0}^{+}\equiv0$ for $\tau$ in $[0,1]$, $\chi_{\mathcal{C},0}^{+}\equiv1$
for $\tau\geq\tau_{0}$ such that
\[
\sup_{[0,\infty)}|(\chi_{\mathcal{C},0}^{+})'|,\,\sup_{[0,\infty)}|(\chi_{\mathcal{C},0}^{+})''|\leq1.
\]
Then defining 
\[
\chi_{\mathcal{C},n}^{+}\left(t\right)\eqdf\begin{cases}
0 & \text{for \ensuremath{t\in\left[0,1\right]} }\\
\chi_{\mathcal{C},0}^{+}\left(\frac{\kappa\left(n\right)}{60}\left(t-1\right)+1\right) & \text{for }t\in\left(1,\infty\right)
\end{cases},
\]
we have
\[
\sup_{[0,\infty)}|(\chi_{\mathcal{C},n}^{+})'|,\,\sup_{[0,\infty)}|(\chi_{\mathcal{C},n}^{+})''|\leq\frac{\kappa\left(n\right)}{60}.
\]
Note that $\chi_{\mathcal{C},n}^{+}\left(\tau\right)\equiv1$ for
$\tau\geqslant\tau_{n}\eqdf\frac{60}{\kappa\left(n\right)}\left(\tau_{0}-1\right)+1$.
The calculation in \cite[Lemma 4.1]{HM23} gives
\[
\|\nabla\chi_{\mathcal{C},n}^{+}\|_{\infty}=\sup_{[0,\infty)}|(\chi_{\mathcal{C},n}^{+})'|\leqslant\frac{\kappa\left(n\right)}{30},
\]
and 
\[
\|\Delta\chi_{\mathcal{C},n}^{+}\|_{\infty}=\sup_{[0,\infty)}|(\chi_{\mathcal{C},n}^{+})''-(\chi_{\mathcal{C},n}^{+})'|\leqslant\frac{\kappa\left(n\right)}{30}.
\]
 If one chooses $\chi_{\mathcal{C},n}^{-}$ to be a function with
$\chi_{\mathcal{C}}^{-}(\tau)$ $\equiv0$ for $\tau\leq\tau_{n}$
and $\chi_{\mathcal{C}}^{-}(\tau)\equiv1$ for $\tau\geq2\tau_{n}$,
(\ref{eq:stagger}) is satisfied and the lemma is proved.
\end{proof}
We obtain the operators
\[
\chi_{\mathcal{C},n,\phi}^{\pm}:L^{2}\left(X;E_{\phi}\right)\to L^{2}\left(X;E_{\phi}\right),
\]
in the unitary case by multiplication by $\chi_{\mathcal{C},n}^{\pm}$.
For the case of covers, we lift $\chi_{\mathcal{C},n}^{\pm}$ to $X_{\varphi}$
via the covering map to obtain a function $\chi_{\mathcal{C},n,\varphi}^{\pm}$
in $L^{2}\left(X_{\varphi}\right)$ and view $\chi_{\mathcal{C},n,\varphi}^{\pm}$
as a multiplication operator
\[
\chi_{\mathcal{C},n,\varphi}^{\pm}:L^{2}\left(X_{\varphi}\right)\to L^{2}\left(X_{\varphi}\right)
\]
We extend $\mathcal{C}$ to the parabolic cylinder 
\[
\tilde{\mathcal{C}}\eqdf\left(0,\infty\right)\times S^{1},
\]
with the same metric. Letting $\mathcal{C}_{\varphi}$ denote the
subset of $X_{\varphi}$ that covers $\mathcal{C}$, we let $\tilde{\mathcal{C}}_{\varphi}$
be the corresponding extension of $C_{\varphi}$. We consider the
Laplacian 
\[
\Delta_{\tilde{\mathcal{C}}_{\varphi}}:H_{\text{new}}^{2}\left(\tilde{\mathcal{C}}_{\varphi}\right)\to L_{\text{new}}^{2}\left(\tilde{\mathcal{C}}_{\varphi}\right).
\]
Given $\phi\in\text{Hom}\left(\mathbb{Z};\text{U}\left(n\right)\right)$
we consider the associated unitary bundle $E_{\phi,\mathcal{\tilde{C}}}\to\tilde{\mathcal{C}}$
with the Laplacian 
\[
\Delta_{\phi,\tilde{\mathcal{C}}}:H^{2}\left(\tilde{\mathcal{C}};E_{\phi,\mathcal{\tilde{C}}}\right)\to L^{2}\left(\tilde{\mathcal{C}};E_{\phi,\mathcal{\tilde{C}}}\right).
\]
By \cite[Lemma 4.2]{HM23} and \cite[Lemma 3.1]{Za22} (see also \cite[Proposition 4.16]{DFP21}),
we have that the corresponding resolvents,
\begin{align*}
R_{\tilde{C},P,\varphi}\left(s\right) & \eqdf\left(\Delta_{\tilde{\mathcal{C}}_{\varphi}}-s(1-s)\right)^{-1}:L_{\text{new}}^{2}\left(\tilde{\mathcal{C}}_{\varphi}\right)\to H_{\text{new}}^{2}\left(\tilde{\mathcal{C}}_{\varphi}\right),\\
R_{\tilde{C},U,\phi}\left(s\right) & \eqdf\left(\Delta_{\phi,\tilde{\mathcal{C}}}-s(1-s)\right)^{-1}:L^{2}\left(\tilde{\mathcal{C}};E_{\phi,\mathcal{\tilde{C}}}\right)\to H^{2}\left(\tilde{\mathcal{C}};E_{\phi,\mathcal{\tilde{C}}}\right),
\end{align*}
exist as bounded operators for $\text{Re}\left(s\right)>\frac{1}{2}$
with 
\[
\|R_{\tilde{C},P,\varphi}\left(s\right)\|_{L_{\text{new}}^{2}\left(\tilde{\mathcal{C}}_{\varphi}\right)},\text{}\|R_{\tilde{C},P,\varphi}\left(s\right)\|_{L_{\text{new}}^{2}\left(\tilde{\mathcal{C}}_{\varphi}\right)}\leqslant\frac{5}{4\kappa\left(n\right)},
\]
and 
\[
\|R_{\tilde{C},U,\phi}\left(s\right)\|_{L^{2}\left(\tilde{\mathcal{C}};E_{\phi,\mathcal{\tilde{C}}}\right)},\text{}\|\Delta R_{\tilde{C},U,\phi}\left(s\right)\|_{L^{2}\left(\tilde{\mathcal{C}};E_{\phi,\mathcal{\tilde{C}}}\right)}\leqslant\frac{5}{4\kappa\left(n\right)},
\]
for $s\in\left[\frac{1}{2}+\sqrt{\kappa\left(n\right)},1\right]$.
Precisely as in \cite{HM23}, we define the cusp parametrix for a
cover $X_{\varphi}$ by
\begin{align*}
 & \mathbb{M}_{P,\varphi}^{\text{cusp}}(s)\eqdf\chi_{\mathcal{C},n,\varphi}^{+}R_{\tilde{C},\varphi}\left(s\right)\chi_{\mathcal{C},n,\varphi}^{-}\\
 & \mathbb{M}_{P,\varphi}^{\text{cusp}}(s):L_{\text{new}}^{2}\left(X_{\varphi}\right)\to H_{\text{new}}^{2}\left(X_{\varphi}\right).
\end{align*}
Here we view $\chi_{\mathcal{C},n,\varphi}^{-}:L_{\text{new}}^{2}\left(X_{\varphi}\right)\to L_{\text{new}}^{2}\left(\tilde{\mathcal{C}}_{\varphi}\right)$
and $\chi_{\mathcal{C},n,\varphi}^{+}:H_{\text{new}}^{2}\left(\tilde{\mathcal{C}}_{\varphi}\right)\to H_{\text{new}}^{2}\left(X_{\varphi}\right)$
in the natural way. As in \cite{Za22}, we analogously define the
cusp parametrix for unitary bundles $E_{\phi}$ as 
\begin{align}
 & \mathbb{M}_{U,\phi}^{\text{cusp}}(s):L^{2}\left(X;E_{\phi}\right)\to L^{2}\left(X;E_{\phi}\right)\label{eq:mcusp-def}\\
 & \mathbb{M}_{U,\phi}^{\text{cusp}}(s)\eqdf\chi_{\mathcal{C},n,\phi}^{+}R_{\tilde{C},\phi}\left(s\right)\chi_{\mathcal{C},n,\phi}^{-}.\nonumber 
\end{align}
We have
\begin{align}
(\Delta-s(1-s))\mathbb{M}_{P,\varphi}^{\text{cusp}}(s) & =\chi_{\mathcal{C},n}^{-}+[\Delta,\chi_{\mathcal{C},n,\varphi}^{+}]R_{\tilde{C}}\left(s\right)\chi_{\mathcal{C},n}^{-}\nonumber \\
 & =\chi_{\mathcal{C},n,\varphi}^{-}+\mathbb{L}_{P,\varphi}^{\text{cusp}}(s)\label{eq:Cusp-para-calc}
\end{align}
 where 
\[
\mathbb{L}_{P,\varphi}^{\text{cusp}}(s)\eqdf[\Delta,\chi_{\mathcal{C},n,\varphi}^{+}]R_{\tilde{C},\varphi}\left(s\right)\chi_{\mathcal{C},n,\varphi}^{-}.
\]
Similarly 
\[
(\Delta-s(1-s))\mathbb{M}_{U,\phi}^{\text{cusp}}(s)=\chi_{\mathcal{C},n}^{-}+\mathbb{L}_{U,\phi}^{\text{cusp}}(s),
\]
where 
\[
\mathbb{L}_{U,\phi}^{\text{cusp}}(s)\eqdf[\Delta,\chi_{\mathcal{C},n,\phi}^{+}]R_{\tilde{C},\phi}\left(s\right)\chi_{\mathcal{C},n,\phi}^{-}.
\]
By Lemma,\ref{Bounding-cutoff-derivative}, it follows by \cite[Lemma 4.3]{HM23}
(or \cite[Lemma 3.2]{Za22} for the unitary case) that for $s\in\left[\frac{1}{2}+\sqrt{\kappa\left(n\right)},1\right]$,
\begin{align}
\|\mathbb{L}_{P,\varphi}^{\text{cusp}}(s)\|_{L_{\text{new}}^{2}\left(X_{\varphi}\right)},\text{ }\|\mathbb{L}_{U,\phi}^{\text{cusp}}(s)\|_{L^{2}\left(X;E_{\phi}\right)} & \leq\left(\|(\Delta\chi_{\mathcal{C},n}^{+})\|_{\infty}+2\|\nabla\chi_{\mathcal{C},n}^{+}\|_{\infty}\right)\cdot\frac{5}{4\kappa\left(n\right)}\leqslant\frac{1}{8}.\label{eq:cusp-norm-bound}
\end{align}
This deterministic bound on the cusp parametrix will be sufficient
for our purposes.

\subsection{\label{subsec:Operators-on}Operators on $\mathbb{H}$}

For $s\in\mathbb{C}$ with $\text{Re}(s)>\frac{1}{2}$, let
\[
R_{\mathbb{H}}(s):L^{2}\left(\mathbb{H}\right)\to L^{2}\mathbb{\left(H\right)},\text{}R_{\mathbb{H}}(s)\eqdf\left(\Delta_{\mathbb{H}}-s(1-s)\right)^{-1},
\]
be the resolvent on the upper half plane. Then $R_{\mathbb{H}}(s)$
is an integral operator with radial kernel $R_{\mathbb{H}}(s;r)$.
Let $\chi_{0}:\mathbb{R}\to\left[0,1\right]$ be a smooth function
such that 
\[
\chi_{0}\left(t\right)=\begin{cases}
1 & \text{if \ensuremath{t\leqslant0}, }\\
0 & \text{if \ensuremath{t\geqslant1}. }
\end{cases}.
\]
For $T>0$, we define a smooth cutoff function $\chi_{T}$ by 
\[
\chi_{T}(t)\eqdf\chi_{0}(t-T).
\]
We then define the operator $R_{\mathbb{H}}^{(T)}(s):L^{2}\left(\mathbb{H}\right)\to L^{2}\left(\mathbb{H}\right)$
to be the integral operator with radial kernel
\[
R_{\mathbb{H}}^{(T)}(s;r)\eqdf\chi_{T}(r)R_{\mathbb{H}}(s;r).
\]
It is proved in \cite[Lemma 5.3]{HM23} that for any $f\in C_{c}^{\infty}\left(\mathbb{H}\right)$
and $s\in\left[\frac{1}{2},1\right]$, we have 
\begin{enumerate}
\item $R_{\mathbb{H}}^{(T)}(s)f\in H^{2}\left(\mathbb{H}\right).$
\item $\left(\Delta-s(1-s)\right)R_{\mathbb{H}}^{(T)}(s)f=f+\mathbb{L}_{\mathbb{H}}^{(T)}(s)f$
as equivalence classes of $L^{2}$ functions where $\mathbb{L}_{\H}^{(T)}(s)$
is defined to be the integral operator with radial kernel
\begin{equation}
\mathbb{L}_{\H}^{(T)}(s;r_{0})\eqdf\left(-\frac{\partial^{2}}{\partial r^{2}}[\chi_{T}]-\frac{1}{\tanh r}\frac{\partial}{\partial r}[\chi_{T}]\right)R_{\H}(s;r_{0})-2\frac{\partial}{\partial r}[\chi_{T}]\frac{\partial R_{\H}}{\partial r}(s;r_{0}).\label{eq:remainder-kernel-def}
\end{equation}
\end{enumerate}
We recall some important properties of $\mathbb{L}_{\mathbb{H}}^{(T)}\left(s;r\right)$
from \cite[Lemma 5.1]{HM23}.
\begin{lem}
\label{lem:LH-bounds}We have 
\begin{enumerate}
\item For $T>0$ and $s\in[\frac{1}{2},1]$, $\mathbb{L}^{(T)}(s;\bullet)$
is smooth and supported in $\left[T,T+1\right]$.
\item There is a constant $C>0$ such that for any $T>0$ and $s\in[\frac{1}{2},1]$
we have 
\[
|\mathbb{L}_{\H}^{(T)}(s;r_{0})|\leq Ce^{-sr_{0}}
\]
\item \label{enu:s-deriv}There is a constant $C>0$ such that for any $T>0$,
$s\in\left[\frac{1}{2},1\right]$ and $r_{0}\in\left[T,T+1\right]$
\[
\left|\frac{\partial\mathbb{L}_{\H}^{(T)}}{\partial s}(s_{0};r_{0})\right|\leq C.
\]
\end{enumerate}
\end{lem}

Our goal is to eventually use a Neumann series argument to invert
$\left(\Delta-s(1-s)\right):H_{\text{new}}^{2}\left(X_{\varphi}\right)\to L_{\text{new}}^{2}\left(X_{\varphi}\right)$
which will require an estimate for the operator norm of $\mathbb{L}_{\H}^{(T)}(s)$.
This is given by \cite[Lemma 5.2]{HM23}.
\begin{lem}[{\cite[Lemma 5.2]{HM23} }]
There is a constant $C>0$ such that for any $T>0$ and $s\in[\frac{1}{2},1]$
the operator $\mathbb{L}_{\H}^{(T)}(s)$ extends to a bounded operator
on $L^{2}(\H)$ with operator norm

\[
\|\mathbb{L}_{\H}^{(T)}(s)\|_{L^{2}\left(\mathbb{H}\right)}\leq CTe^{\left(\frac{1}{2}-s\right)T}.
\]
\end{lem}

We will need to ensure that, for example,
\[
\|\mathbb{L}_{\H}^{(T)}(s)\|_{L^{2}\left(\mathbb{H}\right)}<\frac{1}{5},
\]
for $s\in\left[\frac{1}{2}+\sqrt{\kappa(n)},1\right]$. This means
we have to take $T\left(n\right)$ such that, 
\begin{equation}
Te^{-T\sqrt{\kappa\left(n\right)}}<\frac{1}{5}\label{eq:decay-infinity-op}
\end{equation}
 for all sufficiently large $n$. We will eventually take $\kappa\left(n\right)=\frac{4\left(\log T\right)^{2}}{T^{2}}$
which ensures (\ref{eq:decay-infinity-op}).

\subsection{Interior parametrix\label{subsec:Interior-parametrix}}

As in \cite{HM23,Za22}, we define,
\begin{align*}
R_{\H,U,n}^{(T)}(s;x,y)\eqdf R_{\H}^{(T)}(s;x,y)\mathrm{Id}_{n},\text{ } & R_{\H,P,n}^{(T)}(s;x,y)\eqdf R_{\H}^{(T)}(s;x,y)\mathrm{Id}_{V_{n}^{0}},\\
\mathbb{L}_{\H,U,n}^{(T)}(s;x,y)\eqdf\mathbb{L}_{\H}^{(T)}(s;x,y)\mathrm{Id}_{n},\text{ } & \mathbb{L}_{\H,P,n}^{(T)}(s;x,y)\eqdf\mathbb{L}_{\H}^{(T)}(s;x,y)\mathrm{Id}_{V_{n}^{0}},
\end{align*}
and $R_{\H,U,n}^{(T)}(s),R_{\H,P,n}^{(T)}(s),\mathbb{L}_{\H,U,n}^{(T)}(s),\mathbb{L}_{\H,P,n}^{(T)}(s)$
as the corresponding integral operators. The relevant properties are
summarized in the following Lemma. 
\begin{lem}[{\cite[Lemma 5.5]{HM23}}]
\label{lem:parametrix-bounded}For all $s\in[\frac{1}{2},1]$, 
\begin{enumerate}
\item The integral operator $R_{\H,P,n}^{(T)}(s)(1-\chi_{\mathcal{C},n}^{-})$
is well-defined on $C_{c,\varphi}^{\infty}(\H;V_{n}^{0})$ and extends
to a bounded operator 
\[
R_{\H,P,n}^{(T)}(s)\left(1-\chi_{\mathcal{C},n}^{-}\right):L_{\varphi}^{2}\left(\H;V_{n}^{0}\right)\to H_{\varphi}^{2}\left(\H;V_{n}^{0}\right).
\]
\item The integral operator $\mathbb{L}_{\H,P,n}^{(T)}(s)(1-\chi_{\mathcal{C},n}^{-})$
is well-defined on $C_{c,\phi}^{\infty}\left(\H;V_{n}^{0}\right)$
and and extends to a bounded operator on $L_{\varphi}^{2}(\H;V_{n}^{0})$.
\item We have 
\begin{equation}
\left[\Delta-s(1-s)\right]R_{\H,P,\varphi}^{(T)}(s)(1-\chi_{\mathcal{C},n}^{-})=(1-\chi_{\mathcal{C},n}^{-})+\mathbb{L}_{\H,P,n}^{(T)}(s)(1-\chi_{\mathcal{C},n}^{-})\label{eq:automorphic-para-eq}
\end{equation}
 as an identity of operators on $L_{\varphi}^{2}(\H;V_{n}^{0})$.
\end{enumerate}
\end{lem}

The analogous statement holds for $R_{\H,U,n}^{(T)}(s),\mathbb{L}_{\H,U,n}^{(T)}(s)$.
We define our interior parametrix for surfaces $X_{\varphi}$,
\[
\mathbb{M}_{P,\varphi}^{\text{int}}(s):L_{\text{new}}^{2}\left(X_{\varphi}\right)\to H_{\text{new}}^{2}\left(X_{\varphi}\right),
\]
to be the operator corresponding under $L_{\text{new}}^{2}\left(X_{\varphi}\right)\cong L_{\varphi}^{2}\left(\H;V_{n}^{0}\right)$
and $H_{\text{new}}^{2}\left(X_{\varphi}\right)\cong H_{\varphi}^{2}\left(\H;V_{n}^{0}\right)$
to the integral operator $R_{\H,U,n}^{(T)}(s)\left(1-\chi_{\mathcal{C},n}^{-}\right).$
We make the analogous definition for unitary bundles and use the notation
\[
\mathbb{M}_{U,\phi}^{\text{int}}(s):L^{2}\left(X;E_{\phi}\right)\to H^{2}\left(X;E_{\phi}\right).
\]
Then by defining 
\[
\mathbb{M}_{P,\varphi}\left(s\right)=\mathbb{M}_{P,\varphi}^{\text{int}}(s)+\mathbb{M}_{P,\varphi}^{\text{\text{cusp}}}(s),
\]
we obtain, using (\ref{eq:Cusp-para-calc}),
\begin{align}
\left(\Delta_{X_{\varphi}}-s(1-s)\right)\mathbb{M}_{P,\varphi}\left(s\right) & =\left(1-\chi_{\mathcal{C},n,\varphi}^{-}\right)+\mathbb{L}_{P,\varphi}^{\text{int}}(s)+\chi_{\mathcal{C},n}^{-}+\chi_{\mathcal{C},n,\varphi}^{+}R_{\tilde{C},\varphi}\left(s\right)\chi_{\mathcal{C},n,\varphi}^{-}\nonumber \\
 & =1+\mathbb{M}_{P,\varphi}^{\text{int}}(s)+\mathbb{M}_{P,\varphi}^{\text{cusp}}(s).\label{eq:fundamental-identity}
\end{align}
We make analogous definition for the case of unitaries with the notation
$\mathbb{M}_{U,\phi}\left(s\right)$ and (\ref{eq:fundamental-identity})
holds in this context. 

\section{Probabilistic bounds on operator norms\label{sec:5}}

In this section we prove the probabilistic estimates needed for the
proofs of Theorem \ref{thm:Covers} and Theorem \ref{Thm-Unitaries}. 

\subsection{Preliminaries}

Throughout this subsection, let $\kappa:\mathbb{N}\to\left(0,\infty\right)$
be given and let $\chi_{\mathcal{C},n}^{\pm}$ be chosen as to satisfy
the conclusion of Lemma \ref{Bounding-cutoff-derivative}. Eventually
we will take $\kappa\left(n\right)=\frac{64\left(32d+160\right)\left(\log\log n\right)^{2}}{\log n}$
for random unitaries, where $d$ is the rank of the free group $\Gamma$,
and $\kappa\left(n\right)=\frac{4\cdot24^{2}\left(\log\log\log n\right)^{2}}{\log\log n}$
for random covers. The purpose of this subsection is to ensure that
our random operators $\mathbb{M}^{\text{int}}(s)$ are of the correct
form as to apply Corollary \ref{thm:BC-unitaries-finalV} and Corollary
\ref{thm:BC-Permutations-finalV}.

Let $f\in C_{\phi}^{\infty}\left(\H;\mathbb{C}^{n}\right)$ with $\|f\|_{L^{2}(F)}^{2}<\infty$.
We have 
\begin{align}
\mathbb{L}_{\H,U,n}^{(T)}(s)\left(1-\chi_{\mathcal{C},n}^{-}\right)[f](x) & =\int_{y\in\H}\mathbb{L}_{\H,U,n}^{(T)}\left(s;x,y\right)\left(1-\chi_{\mathcal{C},n}^{-}(y)\right)f(y)d\H(y)\nonumber \\
 & =\sum_{\gamma\in\Gamma}\int_{y\in F}\mathbb{L}_{\H,U,n}^{(T)}\left(s;\gamma x,y\right)\rho_{\phi}\left(\gamma^{-1}\right)\left(1-\chi_{\mathcal{C},n}^{-}\left(y\right)\right)f(y)d\H(y).\label{eq:error-inv-form}
\end{align}
We have an isomorphism of Hilbert spaces
\begin{align*}
L_{\phi}^{2}\left(\H;\mathbb{C}^{n}\right) & \cong L^{2}(F)\otimes\mathbb{C}^{n};\\
f\mapsto & \sum_{e_{i}}\langle f\lvert_{F},e_{i}\rangle_{\mathbb{C}^{n}}\otimes e_{i}.
\end{align*}
Conjugating by this isomorphism,
\[
\mathbb{L}_{\H,U,n}^{(T)}(s)\left(1-\chi_{\mathcal{C},n}^{-}\right)\cong\mathcal{L}_{U,n,\phi}(s)\eqdf\sum_{\gamma\in\Gamma}a_{\gamma,n}^{(T)}(s)\otimes\rho_{\phi}\left(\gamma^{-1}\right),
\]
where
\begin{align*}
a_{\gamma,n}^{(T)}(s) & :L^{2}\left(F\right)\to L^{2}\left(F\right)\\
a_{\gamma,n}^{(T)}(s)[f](x) & \eqdf\int_{y\in F}\mathbb{L}_{\H}^{(T)}\left(s;\gamma x,y\right)\left(1-\chi_{\mathcal{C},n}^{-}\left(y\right)\right)d\H(y).
\end{align*}
Note that for any $n\in\mathbb{N},$$T>1$, $s\in\left[\frac{1}{2},1\right]$
and $\gamma\in\Gamma$, the operator $a_{\gamma,n}^{(T)}\left(s\right)$
is an Hilbert-Schmidt operator whose Hilbert-Schmidt norm can be bounded
by a constant which only depends on $X$. Indeed by Lemma \ref{lem:LH-bounds},
we have
\begin{align*}
\int_{x,y\in F}\left|\mathbb{L}_{\H}^{(T)}\left(s;\gamma x,y\right)\left(1-\chi_{\mathcal{C},n}^{-}\left(y\right)\right)\right|^{2}d\H(x)d\H(y) & \leqslant C\text{Vol}\left(X\right)^{2}.
\end{align*}
In precisely the same way, for the case of covers, we have
\[
L_{\phi}^{2}\left(\H;V_{n}^{0}\right)\cong L^{2}(F)\otimes V_{n}^{0}
\]
and 
\[
\mathbb{L}_{\H,P,n}^{(T)}(s)\left(1-\chi_{\mathcal{C},n}^{-}\right)\cong\mathcal{L}_{P,n,\varphi}(s)\eqdf\sum_{\gamma\in\Gamma}a_{\gamma,n}^{(T)}(s)\otimes\rho_{\varphi}\left(\gamma^{-1}\right).
\]
It is crucial that the map $\gamma\mapsto a_{\gamma,n}^{(T)}(s)$
has finite support whose size we can control. 
\begin{lem}
\label{lem:size-of-support}Given $n$ and $T>0$, there is a finite
set $S\left(T\right)\subset\Gamma$ which contains the support of
the map $\gamma\mapsto a_{\gamma,n}^{(T)}(s)$ for any any $s>\frac{1}{2}$.
There is a constant $C>0$ such that 
\begin{equation}
\left|S\left(T\right)\right|\leqslant C\kappa\left(n\right)^{2}e^{2T},\label{eq:S-bound}
\end{equation}
and if $\gamma\in S\left(T\right)$ then its word-length $\textup{wl}\left(\gamma\right)$
satisfies
\begin{equation}
\text{\textup{wl}\ensuremath{\left(\gamma\right)}\ensuremath{\ensuremath{\leqslant}C\ensuremath{\kappa\left(n\right)^{2}e^{2T}}.}}\label{eq:Wl-bound}
\end{equation}
\end{lem}

\begin{proof}
We define 
\[
K_{n}\eqdf\text{Supp}\left(1-\chi_{\mathcal{C},n}^{-}\right)\subset F.
\]
Recall from (\ref{eq:cusp-region}) that $H\left(L\right)$ is the
region of the fundamental domain $F$ with $y\geqslant L$ . By the
definition of $\chi_{\mathcal{C},n}^{-}$ (Section \ref{subsec:Cusp-parametrix}),
we have 
\[
K_{n}\subset F\backslash H\left(\frac{C}{\kappa\left(n\right)}\right),
\]
for some constant. We have that
\[
F\backslash H\left(\frac{C}{\kappa\left(n\right)}\right)=\left(F\backslash H\left(1\right)\right)\bigsqcup\left(H\left(1\right)\backslash H\left(\frac{C}{\kappa\left(n\right)}\right)\right).
\]
The diameter of $\left(F\backslash H\left(1\right)\right)$ is bounded
by a constant depending only on $X$. The diameter of $H\left(1\right)\backslash H\left(\frac{C}{\kappa\left(n\right)}\right)$
is bounded above by $\log\left(\frac{C}{\kappa\left(n\right)}\right)+2$.
It follows that 
\[
\text{diam}\left(K_{n}\right)\leqslant C+\log\left(\frac{1}{\kappa\left(n\right)}\right).
\]
Then for $x\in F$, by Lemma \ref{lem:LH-bounds}, the expression
\[
\mathbb{L}_{\mathbb{H}}^{(T)}\left(s;\gamma x,y\right)\left(1-\chi_{\mathcal{C},n}^{-}\left(y\right)\right)
\]
is non-zero only when $y\in\text{\ensuremath{K_{n}}}$ and $d\left(\gamma x,y\right)\leqslant T+1$.
Recall that $F$ is a Dirichlet domain about some point $w$, we can
assume $w\in K_{n}$. Then 
\begin{align*}
d\left(\gamma x,w\right) & \leqslant d\left(\gamma x,y\right)+d\left(w,y\right)\\
 & \leqslant T+1+\text{diam}\left(K_{n}\right).
\end{align*}
Then since $F$ is a Dirichlet domain about $w$,
\begin{align*}
d\left(\gamma w,w\right) & \leqslant d\left(\gamma w,\gamma x\right)+d\left(\gamma x,w\right)=d\left(w,x\right)+d\left(\gamma x,w\right)\leqslant2d\left(\gamma x,w\right)\\
 & \leqslant2\left(C+\log\left(\frac{1}{\kappa\left(n\right)}\right)+T\right).
\end{align*}
Then we can employ a lattice point count to deduce that 
\begin{align*}
\left|S\left(T\right)\right| & \leqslant\#\left\{ \gamma\in\Gamma\mid d\left(\gamma w,w\right)\leqslant C+2\log\kappa\left(n\right)+2T\right\} \\
 & \leqslant C\exp\left(2\left(C+\log\left(\frac{1}{\kappa\left(n\right)}\right)+T\right)\right)\leqslant C\frac{e^{2T}}{\kappa\left(n\right)^{2}},
\end{align*}
proving (\ref{eq:S-bound}). 

We now show that property (\ref{eq:Wl-bound}) holds. We assumed that
$F$ is a Dirichlet domain for $\Gamma$, we can also assume that
$F$ is such that the set of side pairings $\left\{ h_{1},\dots,h_{k},h_{1}^{-1},\dots,h_{k}^{-1}\right\} $
for $F$ contain our choice of generators $\gamma_{1},\dots,\gamma_{d}$
and their inverses. We let $\overline{\text{wl}}\left(\gamma\right)$
denote the minimal length of $\gamma$ as a word in $\left\{ h_{1},\dots,h_{k},h_{1}^{-1},\dots,h_{k}^{-1}\right\} $
. Since any $h_{i}$ or its inverse $h_{i}^{-1}$ is a finite word
in $\gamma_{1},\dots,\gamma_{d},\gamma_{1}^{-1},\dots,\gamma_{d}^{-1}$
it follows that there is a constant $C>0$ with
\[
\text{wl}\left(\gamma\right)\leqslant C\overline{\text{wl}}\left(\gamma\right).
\]
We now set about bounding 
\[
\sup_{\gamma\in S\left(T\right)}\overline{\text{wl}}\left(\gamma\right).
\]
By the previous argument, if $\gamma\in S\left(T\right)$ then 
\begin{equation}
\gamma F\cap B\left(w,\text{diam}\left(K_{n}\right)+T+1\right)\neq\emptyset.\label{eq:intersectionsF}
\end{equation}
We claim that if $\gamma$ satisfies (\ref{eq:intersectionsF}) and
$\overline{\text{wl}}\left(\gamma\right)\geqslant1$, then there is
a $\gamma'$ with $\overline{\text{wl}}\left(\gamma\right)=\overline{\text{wl}}\left(\gamma'\right)-1$
which satisfies (\ref{eq:intersectionsF}). The case $\overline{\text{wl}}\left(\gamma\right)=1$
is clear since $w\in F$. For $l>1$ let $\Gamma_{l}$ denote the
elements of $\Gamma$ with $\overline{\text{wl}}\left(\gamma\right)=l$.
Since $\left\{ h_{1},\dots,h_{k},h_{1}^{-1},\dots,h_{k}^{-1}\right\} $
are side pairings for the Dirichlet domain $F$, we see that see that
\[
\bigcup_{\gamma\in\Gamma}\gamma F\backslash\left(\bigcup_{\gamma\in\Gamma_{l}}\gamma F\right)=\left(\bigcup_{i<l}\bigcup_{\gamma\in\Gamma_{i}}\gamma F\right)^{\circ}\sqcup\left(\bigcup_{i>l}\bigcup_{\gamma\in\Gamma_{i}}\gamma F\right)^{\circ}.
\]
 is disconnected. Here $U^{\circ}$denotes the interior of $U$. Therefore
if there claim were not true, then one could find an $l>1$ with 
\begin{equation}
\bigcup_{\gamma\in\Gamma_{l}}\gamma F\cap B\left(w,\text{diam}\left(K_{n}\right)+T+1\right)\neq\emptyset,\label{eq:contradiction}
\end{equation}
such that 
\[
\bigcup_{\gamma\in\Gamma_{l-1}}F\cap B\left(w,\text{diam}\left(K_{n}\right)+T+1\right)=\emptyset,
\]
in particular,
\[
B\left(w,\text{diam}\left(K_{n}\right)+T+1\right)\subset\bigcup_{\gamma\in\Gamma}\gamma F\backslash\left(\bigcup_{\gamma\in\Gamma_{l-1}}\gamma F\right).
\]
Then since the ball of radius $r$ in the hyperbolic plane is connected
and the identity in $\Gamma$ satisfies (\ref{eq:intersectionsF}),
\[
B\left(w,\text{diam}\left(K_{n}\right)+T+1\right)\subset\left(\bigcup_{i<l-1}\bigcup_{\gamma\in\Gamma_{i}}\gamma F\right)^{\circ}.
\]
This gives a contradiction to (\ref{eq:contradiction}) and the claim
follows. It follows that if $\gamma$ satisfies (\ref{eq:intersectionsF})
then $\overline{\text{wl}}\left(\gamma\right)$ is bounded above by
the number of $\gamma\in\Gamma$ which satisfy (\ref{eq:intersectionsF}).
Then by the argument that led to (\ref{lem:size-of-support}),
\begin{align*}
\sup_{\gamma\in S\left(T\right)}\text{wl}\left(\gamma\right) & \leqslant C\#\left\{ \gamma\in\Gamma\mid\gamma F\cap B\left(w,\text{diam}\left(K_{n}\right)+T+1\right)\neq\emptyset\right\} \leqslant C\frac{e^{2T\left(n\right)}}{\kappa\left(n\right)^{2}},
\end{align*}
and the claim is proved.
\end{proof}
Currently, our operators $\sum_{\gamma\in S}a_{\gamma,n}^{(T)}(s)\otimes\rho_{\phi}\left(\gamma^{-1}\right)$
whose norm we wish to bound are almost of the form of Corollary \ref{thm:BC-unitaries-finalV}
except $a_{\gamma,n}^{(T)}(s):L^{2}\left(F\right)\to L^{2}\left(F\right)$
is not a matrix. However each $a_{\gamma,n}^{(T)}(s)$ is compact
so can be approximated by finite rank operators. We need an effective
version of this whilst having control over the rank in terms of the
error.
\begin{lem}
\label{lem:finite-rank-approx}Let $s\in\left[\frac{1}{2},1\right]$
be given. For every $n\in\mathbb{N}$ and $T>1$, there exists a finite
dimensional subspace $W\subset L^{2}\left(X\right)$ with $\left|W\right|\leqslant C\left(S\left(T\right)\right)^{3}$
for some constant $C$ and finite rank operators $b_{\gamma,n}^{(T)}:W\to W$
for each $\gamma\in S\left(T\right)$ such that 
\[
\|b_{\gamma,n}^{(T)}(s)-a_{\gamma,n}^{(T)}(s)\|_{L^{2}(F)}\leqslant\frac{1}{20|S(T)|},
\]
\end{lem}

\begin{proof}
Let $\gamma\in S\left(T\right)$, then since $a_{\gamma}^{(T)}(s)$
is compact, it has a singular value decomposition
\[
a_{\gamma,n}^{(T)}(s)=\sum_{i\in\mathbb{N}}s_{n}\left(a_{\gamma,n}^{(T)}(s)\right)\langle\cdot,e_{i}\rangle f_{i},
\]
where $\left\{ e_{i}\right\} _{i\in\mathbb{N}}$ and $\left\{ f_{i}\right\} _{i\in\mathbb{N}}$
are orthonormal systems in $L^{2}\left(F\right)$ and $\left\{ s_{n}\right\} _{n\in\mathbb{N}}$
is a decreasing sequence of non-negative real numbers \cite[Theorem VI.17]{RS81}.
The singular values $\left\{ s_{n}\left(A\right)\right\} _{n\in\mathbb{N}}$
of a compact operator $A$ on an Hilbert space $\mathcal{H}$ are
the eigenvalues of $|A|=\sqrt{AA^{*}}$ and if $A$ is Hilbert-Schmidt
then $\|A\|_{\text{H.S.}}^{2}=\sum_{j\in\mathbb{N}}s_{j}\left(A\right)^{2}$. 

By defining 
\[
b_{\gamma,n}^{(T)}(s)\eqdf\sum_{i=1}^{r}s_{i}\left(a_{\gamma}^{(T)}(s)\right)\langle\cdot,e_{i}\rangle f_{i},
\]
we see that $b_{\gamma}^{T}(s):W_{\gamma}\to W_{\gamma}$ where $\left|W_{\gamma}\right|\leqslant2r$
and 
\[
\|b_{\gamma,n}^{(T)}(s)-a_{\gamma,n}^{(T)}(s)\|\leqslant s_{r+1}\left(A\right).
\]
We want $r$ to be such that 
\begin{equation}
s_{r+1}\left(a_{\gamma,n}^{(T)}(s)\right)\leqslant\frac{1}{20|S(T)|}.\label{eq:singularvalueaprox}
\end{equation}
We have 
\[
\sum_{i=1}^{\infty}s_{i}\left(a_{\gamma,n}^{(T)}(s)\right)^{2}=\|a_{\gamma,n}^{(T)}(s)\|_{\text{H.S}}^{2}\leqslant C,
\]
 Then
\begin{align*}
0\leqslant\sum_{i=r+1}^{\infty}s_{i}\left(a_{\gamma,n}^{(T)}(s)\right)^{2} & =\|a_{\gamma,n}^{(T)}(s)\|_{\text{H.S}}^{2}-\sum_{i=1}^{r}s_{i}\left(a_{\gamma,n}^{(T)}(s)\right)^{2}\\
 & \leqslant C-rs_{r}\left(a_{\gamma,n}^{(T)}(s)\right)^{2}.
\end{align*}
In particular,
\[
s_{r}\left(a_{\gamma,n}^{(T)}(s)\right)\leqslant\sqrt{\frac{C}{r}}.
\]
Taking $r\geqslant400\cdot C\cdot S\left(T\right)^{2}$ guarantees
that (\ref{eq:singularvalueaprox}) is satisfied. Then $\left|W_{\gamma}\right|\leqslant CS\left(T\right)^{2}$
for each $\gamma\in S\left(T\right)$ and taking 
\[
W=\bigcup_{\gamma\in S\left(T\right)}W_{\gamma},
\]
gives the conclusion.
\end{proof}
Finally we prove a simple deviations bound.
\begin{lem}
\label{lem:deviations}There exists a constant $c_{3}>0$ depending
only on $X$ such that for any $T>1$, any $\gamma\in S\left(T\right)$
and $s_{1},s_{2}\in\left[\frac{1}{2},1\right],$
\[
\|a_{\gamma,n}^{(T)}(s_{1})-a_{\gamma,n}^{(T)}(s_{2})\|_{L^{2}(F)}\leq c_{3}|s_{1}-s_{2}|.
\]
\end{lem}

\begin{proof}
The operator 
\[
a_{\gamma,n}^{(T)}(s_{1})-a_{\gamma,n}^{(T)}(s_{2})
\]
 is an integral operator with kernel 
\[
\left(\mathbb{L}_{\H}^{(T)}\left(s;\gamma x,y\right)-\mathbb{L}_{\H}^{(T)}\left(s;\gamma x,y\right)\right)\left(1-\chi_{\mathcal{C},n}^{-}\left(y\right)\right).
\]
We have for any $T>1$, $\gamma\in S\left(T\right)$, by Lemma \ref{lem:parametrix-bounded},
\begin{align*}
\left|\mathbb{L}_{\H}^{(T)}\left(s;\gamma x,y\right)-\mathbb{L}_{\H}^{(T)}\left(s;\gamma x,y\right)\right| & \leqslant\sup_{s\in\left[\frac{1}{2},1\right]}\left|\frac{\partial}{\partial s}\mathbb{L}_{\H}^{(T)}\left(s;\gamma x,y\right)\right|\left|s_{1}-s_{2}\right|\\
 & \leqslant C\left|s_{1}-s_{2}\right|.
\end{align*}
Then we see
\begin{align*}
\|a_{\gamma}^{(T)}(s_{1})-a_{\gamma}^{(T)}(s_{2})\|_{L^{2}\left(F\right)} & \leqslant\|a_{\gamma}^{(T)}(s_{1})-a_{\gamma}^{(T)}(s_{2})\|_{\HS}\leqslant c_{3}\left|s_{1}-s_{2}\right|,
\end{align*}
for some constant $c_{3}>0$.
\end{proof}

\subsection{Random operator bounds}

We are now in a position to apply the results of Section \ref{sec:Random-matrix-theory}
to our random operators $\mathcal{L}_{U,n,\phi}(s)$ and $\mathcal{L}_{P,n,\varphi}(s)$. 
\begin{lem}
\label{lem:Radnomopbound}With notations as above,
\begin{itemize}
\item Taking $T=\frac{\sqrt{\log n}}{4\sqrt{32d+160}}$ and $\kappa\left(n\right)=\frac{64\left(32d+160\right)\left(\log\log n\right)^{2}}{\log n}$,
we have that with probability tending to $1$ as $n\to\infty$,
\[
\sup_{s\in\left[\frac{1}{2}+\sqrt{\kappa\left(n\right)},1\right]}\|\mathcal{L}_{U,n,\phi}(s)\|_{L^{2}\left(F\right)\otimes\mathbb{C}^{n}}<\frac{3}{5}.
\]
\item Taking $T=\frac{\sqrt{\log\log n}}{24}$ and $\kappa\left(n\right)=\frac{4\cdot24^{2}\left(\log\log\log n\right)^{2}}{\log\log n}$,
we have that with probability tending to $1$ as $n\to\infty$
\[
\sup_{s\in\left[\frac{1}{2}+\sqrt{\kappa\left(n\right)},1\right]}\|\mathcal{L}_{P,n,\varphi}(s)\|_{L^{2}\left(F\right)\otimes V_{n}^{0}}<\frac{3}{5}.
\]
\end{itemize}
\end{lem}

\begin{proof}
We first treat the unitary case. Let $T=\frac{\sqrt{\log n}}{4\sqrt{32d+160}}$,
$\kappa\left(n\right)=\frac{64\left(32d+160\right)\left(\log\log n\right)^{2}}{\log n}$
and let $s\in\left[\frac{1}{2}+\sqrt{\kappa\left(n\right)},1\right]$
be fixed. Then by Lemma \ref{lem:finite-rank-approx}, there exists
a finite dimensional subspace $W\subset L^{2}\left(X\right)$ with
$m=\left|W\right|\leqslant C\frac{e^{3T}}{\kappa\left(n\right)^{3}}$
and operators $b_{\gamma}^{(T)}:W\to W$ for each $\gamma\in S\left(T\right)$
such that 
\[
\|b_{\gamma,n}^{(T)}(s)-a_{\gamma,n}^{(T)}(s)\|_{L^{2}(F)}\leqslant\frac{1}{20\left|S\left(T\right)\right|}.
\]
It follows that 
\begin{equation}
\|\mathcal{L}_{U,n,\phi}(s)-\sum_{\gamma\in S\left(T\right)}b_{\gamma,n}^{(T)}(s)\otimes\rho_{\phi}\left(\gamma\right)\|_{L^{2}\left(X\right)\otimes\mathbb{C}^{n}}\leqslant\frac{1}{20}.\label{eq:op-to-FR}
\end{equation}
We now apply Corollary \ref{thm:BC-unitaries-finalV} to $\sum_{\gamma\in S\left(T\right)}b_{\gamma,n}^{(T)}(s)\otimes\rho_{\phi}\left(\gamma\right)$. 

By Lemma \ref{lem:size-of-support}, we have that $S\left(T\right)\subset B_{l}$
where $l\leqslant C\frac{e^{2T}}{\kappa\left(n\right)^{2}}$ and $\left|S\left(T\right)\right|\leqslant C\frac{e^{2T}}{\kappa\left(n\right)^{2}}$.
Recall we made the choices $T=\frac{\sqrt{\log n}}{4\sqrt{32d+160}}$,
$\kappa\left(n\right)=\frac{64\left(32d+160\right)\left(\log\log n\right)^{2}}{\log n}$.
We now check that the assumptions of Corollary \ref{thm:BC-unitaries-finalV}
are satisfied.
\begin{align*}
\log\left(2ml\left|S\right|^{\lceil\log_{2}l\rceil}l^{\left(\lceil\log_{2}l\rceil-1\right)}\right) & \leqslant C+\log m+\log l+\log l\log\left|S\right|+\left(\log l\right)^{2}\\
 & \leqslant C+5T-5\log\kappa\left(n\right)+2\left(2T-2\log\kappa\left(n\right)+\log C\right)^{2}\\
 & \ll T^{2}\ll\log n.
\end{align*}
It follows that for some constant $C$,
\[
2ml\left|S\right|^{\lceil\log_{2}l\rceil}l^{\left(\lceil\log_{2}l\rceil-1\right)}<n^{C}<\exp\left(n^{\frac{1}{32d+160}}\right),
\]
for $n$ large enough and the first assumption of Corollary \ref{thm:BC-unitaries-finalV}
holds. We remark that we didn't need to use the precise choice of
constants in $T$ and $\kappa$ to check the first condition. To check
the second condition, we observe that
\begin{equation}
\log\left(l^{2}\left|S\right|^{\lceil\log_{2}l\rceil}l^{\left(\lceil\log_{2}l\rceil-1\right)}\right)\leqslant9\left(T-\log\kappa\left(n\right)\right)^{2}<12T^{2}<\frac{1}{32d+160}\log n,\label{eq:log-inequality}
\end{equation}
for sufficiently large $n$, and by exponentiating (\ref{eq:log-inequality})
\[
\frac{l^{2}\left|S\right|^{\lceil\log_{2}l\rceil}l^{\left(\lceil\log_{2}l\rceil-1\right)}}{n^{\frac{1}{32d+160}}}\to0,
\]
as $n\to\infty$. We are now in the position to apply Corollary \ref{thm:BC-unitaries-finalV}.
We learn that with probability at least
\[
1-\exp\left(-\sqrt{n}\right),
\]
we have
\begin{align*}
\|\sum_{\gamma\in S}b_{\gamma,n}^{(T)}(s)\otimes\rho_{\varphi}\left(\gamma\right)\|_{\mathbb{C}^{m}\otimes\mathbb{C}^{n}} & \leqslant\|\sum_{\gamma\in S}b_{\gamma,n}^{(T)}(s)\otimes\lambda\left(\gamma\right)\|_{\mathbb{C}^{m}\otimes l^{2}\left(\Gamma\right)}\left(1+\frac{l^{2}\left|S\right|^{\lceil\log_{2}l\rceil}l^{\frac{3}{2}\left(\lceil\log_{2}l\rceil-1\right)}}{n^{\frac{1}{32d+160}}}\right)\\
 & =\|\sum_{\gamma\in S}b_{\gamma,n}^{(T)}(s)\otimes\lambda\left(\gamma\right)\|_{\mathbb{C}^{m}\otimes l^{2}\left(\Gamma\right)}\left(1+o\left(1\right)\right).
\end{align*}
We have an isometric linear isomorphism
\begin{align}
L^{2}\left(F\right)\otimes\ell^{2}\left(\Gamma\right) & \cong L^{2}\left(\H\right),\nonumber \\
f\otimes\delta_{\gamma} & \mapsto f\circ\gamma^{-1},\label{eq:map-iso}
\end{align}
(with $f\circ\gamma^{-1}$ extended by zero from a function on $\gamma F$).
Indeed, let $L_{F}^{2}\left(\mathbb{H}\right)$ denote the square-integrable
functions on $\mathbb{H}$ which are supported on finitely many $\Gamma$-translates
of $F$. Then any $g\in L_{F}^{2}\left(\mathbb{H}\right)$ can be
written as 

\[
g=\sum_{i\in I}g\mid_{\gamma_{i}F}=\sum_{i\in I}g_{i}\circ\gamma_{i}^{-1},
\]
where $g_{i}\in L^{2}\left(F\right)$ and $I\subset\Gamma$ is a finite
set, then the inverse to the map (\ref{eq:map-iso}) is given by
\begin{equation}
g\mapsto\sum_{i\in I}g_{i}\otimes\delta_{\gamma_{i}}.\label{eq:map-iso2}
\end{equation}
One can check that the map (\ref{eq:map-iso2}) is an isometric isomorphism
of $L_{F}^{2}\left(\mathbb{H}\right)$ onto its image, which is dense
in $L^{2}\left(F\right)\otimes l^{2}\left(\Gamma\right).$ Since $L_{F}^{2}\left(\mathbb{H}\right)$
is dense in $L^{2}\left(\mathbb{H}\right)$, the map (\ref{eq:map-iso2})
extends to an isometric isomorphism between $L^{2}\left(\mathbb{H}\right)$
and $L^{2}\left(F\right)\otimes l^{2}\left(\Gamma\right)$.

Under this isomorphism, the operator $\sum_{\gamma\in S}a_{\gamma,n}^{(T)}(s)\otimes\lambda\left(\gamma^{-1}\right)$
is conjugated to 
\[
\mathbb{L}_{\H}^{(T)}(s)\left(1-\chi_{\mathcal{C},n}^{-}\right):L^{2}(\H)\to L^{2}(\H)
\]
from Section $\ref{subsec:Operators-on}$. Since $\left(1-\chi_{\mathcal{C},n}^{-}\right)$
is valued in $[0,1]$, multiplication by it has operator norm $\leq1$
on $L^{2}(\H)$, we see that
\begin{align*}
\|\mathbb{L}_{\H}^{(T)}(s)\left(1-\chi_{\mathcal{C},n}^{-}\right)\|_{L^{2}\left(\mathbb{H}\right)} & \leqslant\|\mathbb{L}_{\H}^{(T)}(s)\|_{L^{2}\left(\mathbb{H}\right)}<CT\left(n\right)e^{-T\left(n\right)\left(\frac{1}{2}-s\right)}.
\end{align*}
Since $s\in\left[\frac{1}{2}+\sqrt{\kappa\left(n\right)},1\right]$
and $\kappa\left(n\right)=\frac{4\left(\log T\left(n\right)\right)^{2}}{T\left(n\right)^{2}}$,
we have
\[
T\left(n\right)e^{-T\left(n\right)\left(\frac{1}{2}-s\right)}\leqslant T\left(n\right)e^{-2\log\left(T\left(n\right)\right)}=o\left(1\right).
\]
Then we have
\begin{equation}
\|\sum_{\gamma\in S}a_{\gamma,n}^{(T)}(s)\otimes\lambda(\gamma^{-1})\|_{L^{2}\left(F\right)\otimes l^{2}\left(\Gamma\right)}<\frac{1}{10},\label{eq:infinity-op-normbound}
\end{equation}
for sufficiently large $n$. By the argument that led to (\ref{eq:op-to-FR}),
we see
\begin{equation}
\|\sum_{\gamma\in S}b_{\gamma,n}^{(T)}\otimes\lambda\left(\gamma^{-1}\right)-\sum_{\gamma\in S}a_{\gamma,n}^{(T)}(s)\otimes\lambda(\gamma^{-1})\|_{\mathbb{C}^{m}\otimes l^{2}\left(\Gamma\right)}<\frac{1}{20}.\label{eq:op-to-Fr-infinity}
\end{equation}
Then by (\ref{eq:op-to-FR}), (\ref{eq:infinity-op-normbound}) and
(\ref{eq:op-to-Fr-infinity}), for our fixed choice of $s$,
\[
\|\mathcal{L}_{U,n,\phi}(s)\|_{L^{2}\left(F\right)\otimes\mathbb{C}^{n}}<\frac{2}{5},
\]
with probability at least 
\[
1-\exp\left(-\sqrt{n}\right).
\]
We now use a finite net argument to control all $s\in\left[\frac{1}{2}+\sqrt{\kappa\left(n\right)},1\right]$
uniformly. Let $\mathcal{Y}$ be a finite set of points in $\left[\frac{1}{2}+\sqrt{\kappa\left(n\right)},1\right]$
so that each point of $\left[\frac{1}{2}+\sqrt{\kappa\left(n\right)},1\right]$
is of distance at most
\[
\frac{1}{5|S\left(T\right)|c_{3}},
\]
from the finite set $\mathcal{Y}$, where $c_{3}$ is the constant
in Lemma \ref{lem:deviations}. We can pick $\mathcal{Y}$ so that
$\left|\mathcal{Y}\right|\leqslant5c_{3}\left|S\left(T\right)\right|\leqslant C\frac{e^{2T}}{\kappa\left(n\right)^{2}}$.
Then by applying an intersection bound, the probability that 
\[
\|\mathcal{L}_{U,n,\phi}(s)\|_{L^{2}\left(F\right)\otimes\mathbb{C}^{n}}<\frac{2}{5}
\]
for $\textit{every point}$ $s\in\mathcal{Y}$ is bounded below by
\begin{equation}
1-C\exp\left(-\sqrt{n}\right)\left|S\left(T\right)\right|\geqslant1-n\exp\left(-\sqrt{n}\right),\label{intersectionbound}
\end{equation}
which tends to $1$ as $n\to\infty$ and 
\[
\sup_{s\in\mathcal{Y}}\|\mathcal{L}_{U,n,\phi}(s)\|_{L^{2}\left(F\right)\otimes\mathbb{C}^{n}}\leqslant\frac{2}{5},
\]
a.a.s. Finally, for $s_{1},s_{2}\in\left[\frac{1}{2}+\sqrt{\kappa\left(n\right)},1\right]$,
\begin{equation}
\mathcal{L}_{U,n,\phi}(s_{1})-\mathcal{L}_{U,n,\phi}(s_{2})=\sum_{\gamma\in S\left(T\right)}\left[a_{\gamma}^{(T)}(s_{1})-a_{\gamma}^{(T)}(s_{2})\right]\otimes\rho_{\phi}\left(\gamma^{-1}\right).\label{eq:difff-opeartor}
\end{equation}
 Then by Lemma \ref{lem:deviations}, for some constant $c_{3}>0$
we have
\[
\|a_{\gamma}^{(T)}(s_{1})-a_{\gamma}^{(T)}(s_{2})\|_{L^{2}(F)}\leq c_{3}|s_{1}-s_{2}|,
\]
for all $\gamma\in S\left(T\right)$ and $s_{1},s_{2}\in[s_{0},1]$.
We see that,
\[
\|\mathcal{L}_{U,n,\phi}(s_{1})-\mathcal{L}_{U,n,\phi}(s_{2})\|_{L^{2}\left(F\right)\otimes\mathbb{C}^{n}}\leq\left|S\left(T\right)\right|c_{3}|s_{1}-s_{2}|.
\]
Then by the choice of $\mathcal{Y}$, it follows that 
\[
\sup_{s\in\mathcal{Y}}\|\mathcal{L}_{U,n,\phi}(s)\|\leqslant\frac{2}{5}\implies\sup_{s\in\left[\frac{1}{2}+\sqrt{\kappa\left(n\right)},1\right]}\|\mathcal{L}_{U,n,\phi}(s)\|\leqslant\frac{3}{5}.
\]
Since the prior happens with probability tending to $1$ as $n\to\infty$,
the first claim is proved.

The argument in the case of random covers is similar, one just needs
to verify that the choices of $\kappa\left(n\right)$ and $T$ allow
the same conclusions as in the unitary case. We want to apply Corollary
\ref{thm:BC-Permutations-finalV}, leading us to require that
\[
2ml\left|S\right|^{\lceil\log_{2}l\rceil}l^{\left(\lceil\log_{2}l\rceil-1\right)}\leqslant n^{\sqrt{\log n}},
\]
and 
\[
l^{2}\left|S\right|^{\lceil\log_{2}l\rceil}l^{\left(\lceil\log_{2}l\rceil-1\right)}\leqslant\left(\log\left(n\right)\right)^{\frac{1}{4}}.
\]
Since $m\leqslant C\frac{e^{3T}}{\kappa\left(n\right)^{3}}$ and $l,\left|S\right|\leqslant C\frac{e^{2T}}{\kappa\left(n\right)^{2}}$,
c.f. Lemma \ref{lem:size-of-support} and Lemma \ref{lem:finite-rank-approx},
it is a simple calculation to check both inequalities are satisfied
if one takes $T=\frac{\sqrt{\log\log n}}{24}$ and $\kappa\left(n\right)=\frac{4\cdot24^{2}\left(\log\log\log n\right)^{2}}{\log\log n}$.
Finally, we just need that 
\[
\frac{1}{\sqrt{n}}\cdot\left|S\left(T\right)\right|\to0,
\]
in order to apply the same intersection bound argument (\ref{intersectionbound}),
which holds by our assumptions on $T$ and $\kappa$.
\end{proof}

\section{Proofs of Theorem \ref{thm:Covers} and Theorem \ref{Thm-Unitaries}}

It is now straightforward to conclude Theorem \ref{thm:Covers} and
Theorem \ref{Thm-Unitaries}. Recall, for the case of unitary bundles,
that 
\[
\mathbb{M}_{U,\phi}(s)\eqdf\mathbb{M}_{U,\phi}^{\text{int}}(s)+\mathbb{M}_{U,\phi}^{\text{cusp}}(s),
\]
then $\mathbb{M}_{U,\phi}(s):L^{2}\left(X;E_{\phi}\right)\to H^{2}\left(X;E_{\phi}\right)$
is a bounded operator and 
\begin{equation}
\left(\Delta_{X_{\phi}}-s(1-s)\right)\mathbb{M}_{U,\phi}(s)=1+\mathbb{L}_{U,\phi}^{\text{int}}(s)+\mathbb{L}_{U,\phi}^{\text{cusp}}(s),\label{eq:Neumann}
\end{equation}
by Section \ref{subsec:Interior-parametrix}. We proved in Lemma \ref{lem:Radnomopbound}
that there is a constant $c_{4}$ (whose precise value can be read
off in Lemma \ref{lem:Radnomopbound}) such that a.a.s. 
\[
\|\mathbb{L}_{U,\phi}^{\text{int}}(s)\|\leqslant\frac{3}{5},
\]
for all $s\in\left[\frac{1}{2}+\sqrt{c_{4}}\frac{\log\log n}{\sqrt{\log n}},1\right]$.
Then since by (\ref{eq:cusp-norm-bound})
\[
\|\mathbb{L}_{U,\phi}^{\text{cusp}}(s)\|\leqslant\frac{1}{8},
\]
we have a.a.s.

\[
\sup_{s\in\left[\frac{1}{2}+\sqrt{c_{4}}\frac{\log\log n}{\sqrt{\log n}},1\right]}\|\mathbb{L}_{\phi}^{\text{int}}(s)+\mathbb{L}_{\phi}^{\text{cusp}}(s)\|\leqslant\frac{4}{5}.
\]
This implies that a.a.s.
\[
\mathbb{M}_{U,\phi}(s)\left(1+\mathbb{L}_{U,\phi}^{\text{int}}(s)+\mathbb{L}_{U,\phi}^{\text{cusp}}(s\right)^{-1}
\]
exists as a bounded operator $L^{2}\left(X;E_{\phi}\right)\to H^{2}\left(X;E_{\phi}\right)$
for every $s\in\left[\frac{1}{2}+\sqrt{c_{4}}\frac{\log\log n}{\sqrt{\log n}},1\right]$.
Then by (\ref{eq:Neumann}), we have that a.a.s.

\[
\inf\spec\left(\Delta_{\phi}\right)\geqslant\frac{1}{4}-c_{4}\frac{\left(\log\log n\right)}{\log n}^{2}.
\]
To conclude Theorem \ref{thm:Covers}, we apply precisely the same
argument, using Lemma \ref{lem:Radnomopbound}, to conclude that there
is a constant $c_{5}>0$ such that a.a.s 
\[
\mathbb{M}_{P,\varphi}(s)\left(1+\mathbb{L}_{P,\varphi}^{\text{int}}(s)+\mathbb{L}_{P,\varphi}^{\text{cusp}}(s)\right)^{-1},
\]
exists as a bounded operator $L_{\text{\text{new}}}^{2}\left(X_{\varphi}\right)\to H_{\text{\text{new}}}^{2}\left(X_{\varphi}\right)$
for all $s\in\left[\frac{1}{2}+\sqrt{c_{5}}\frac{\log\log\log n}{\sqrt{\log\log n}},1\right].$
Then
\[
\spec\left(\Delta_{X_{n}}\right)\cap\left[0,\frac{1}{4}-c_{5}\frac{\left(\log\log\log n\right)^{2}}{\log\log n}\right)=\spec\left(\Delta_{X}\right)\cap\left[0,\frac{1}{4}-c_{5}\frac{\left(\log\log\log n\right)^{2}}{\log\log n}\right),
\]
a.a.s.

\lyxaddress{Will Hide, \\
Department of Mathematical Sciences,\\
Durham University, \\
Lower Mountjoy, DH1 3LE Durham,\\
United Kingdom\\
\texttt{william.hide@durham.ac.uk}\\
}

\begin{thebibliography}{GMST21}
\bibitem[AM23]{AM23}N. Anantharaman and L. Monk, Friedman-Ramanujan
functions in random hyperbolic geometry and application to spectral
gaps. Preprint, arXiv:2304.02678. (2023)

\bibitem[Bo20]{Bo20}C. Bordenave. A new proof of Friedman\textquoteright s
second eigenvalue theorem and its extension to random lifts. Ann.
Sci. Ec. Norm. Supér. (4) , 53(6):1393--1439, (2020).

\bibitem[BC19]{BC19}C. Bordenave and B. Collins. Eigenvalues of random
lifts and polynomials of random permutation matrices. Annals of Mathematics,
190(3):811--875, (2019).

\bibitem[BC23]{BC23} C. Bordenave and B. Collins, Norm of matrix-valued
polynomials in random unitaries and permutations, Preprint, arXiv:2304.05714
(2023).

\bibitem[BM04]{BM04}R. Brooks and E. Makover. Random construction
of Riemann surfaces. J. Differential Geom. 68(1):121--157, (2004).

\bibitem[BCP21]{BCP21}T. Budzinski, N. Curien, and B. Petri. The
diameter of random Belyi surfaces. Algebr. Geom. Topol., 21(6):2929--2957,
2021.

\bibitem[Bu84]{Bu84}P. Buser, On the bipartition of graphs, Discrete
Appl. Math. 9 (1984), no. 1, 105--109. 

\bibitem[BBD88]{BBD88}P. Buser, M. Burger, and J. Dodziuk, Riemann
surfaces of large genus and large $\lambda_{1}$, Geometry and analysis
on manifolds (Katata/Kyoto, 1987), Lecture Notes in Math., vol. 1339,
Springer, Berlin, (1988), pp. 54--63.

\bibitem[Di69]{Di69}J. D. Dixon, The probability of generating the
symmetric group, Math. Z. 110 (1969), 199--205.

\bibitem[DFP21]{DFP21}M. Doll, K. Fedosova and A. Pohl, Counting
Resonances on Hyperbolic Surfaces with Unitary Twists, Preprint, arXiv:2109.12923.
(2021)

\bibitem[Fr03]{Fr03}J. Friedman, Relative expanders or weakly relatively
Ramanujan graphs, Duke Math. J. 118 (2003), no. 1, 19--35.

\bibitem[Fr08]{Fr08}J. Friedman. A proof of Alon\textquoteright s
second eigenvalue conjecture and related problems. Mem. Amer. Math.
Soc., 195(910):viii+100, (2008).

\bibitem[Ga06]{Ga06}A. Gamburd. Poisson-Dirichlet distribution for
random Belyi surfaces. Ann. Probab., 34(5):1827--1848, (2006).

\bibitem[GJ78]{GJ78}S. Gelbart and H. Jacquet. A relation between
automorphic representations of $\text{GL}_{2}$ and $\text{GL}_{3}$,
Ann. Sci. Ecole Norm. Sup. (4) 11 (1978), no. 4, 471--542.

\bibitem[GMST21]{GMST21} C. Gilmore, E. Le Mason, T. Sahlsten and
J. Thomas. Short geodesic loops and $L^{p}$ norms of eigenfunctions
on large genus random surfaces. Geom. Funct. Anal. 31, 62--110 (2021).

\bibitem[GPY11]{GPY11}L. Guth, H. Parlier, and R. Young, Pants decompositions
of random surfaces, Geometric and Functional Analysis 21 (2011), 1069--1090.

\bibitem[HT05]{HT05}U. Haagerup and S. Thorbjørnsen, A new application
of random matrices: ${\rm Ext}(C_{{\rm red}}^{*}(F_{2}))$ is not
a group, Ann. of Math. (2) 162, 2 (2005), pp. 711-{}-775.

\bibitem[HM23]{HM23}W. Hide and M. Magee, Near optimal spectral gaps
for hyperbolic surfaces, Ann. of Math. (2) 198(2): 791-824 (2023).

\bibitem[Hi22]{Hi22}W. Hide, Spectral Gap for Weil--Petersson Random
Surfaces with Cusps, International Mathematics Research Notices, (2022).

\bibitem[HT22]{HT22}W. Hide and J. Thomas. Short geodesics and small
eigenvalues on random hyperbolic punctured spheres, Preprint, arXiv:2209.15568
(2022).

\bibitem[HY21]{HY21}J. Huang and H.-T. Yau. Spectrum of random d-regular
graphs up to the edge. arXiv:2102.00963, (2021).

\bibitem[Iw89]{Iw89}H. Iwaniec, Selberg\textquoteright s lower bound
of the first eigenvalue for congruence groups, Number theory, trace
formulas and discrete groups (Oslo, 1987), Academic Press, Boston,
MA, (1989), pp. 371--375.

\bibitem[Iw96]{Iw96}H. Iwaniec, The lowest eigenvalue for congruence
groups, Topics in geometry, Progr. Nonlinear Differential Equations
Appl., vol. 20, Birkhäuser Boston, Boston, MA, (1996), pp. 203--212.

\bibitem[Ki03]{Ki03} H. Kim. Functoriality for the exterior square
of $\text{GL}_{4}$ and the symmetric fourth of $\text{GL}_{2}$,
J. Amer. Math. Soc. 16 (2003), no. 1, 139--183, With appendix 1 by
D.Ramakrishnan and appendix 2 by H. Kim and P. Sarnak.

\bibitem[KS02]{KS02} H. Kim and F. Shahidi. Functorial products for
$\text{GL}_{2}\times\text{GL}_{3}$ and the symmetric cube for GL$_{2}$,
Ann. of Math. (2) 155 (2002), no. 3, 837--893, With an appendix by
C.J. Bushnell and G. Henniart.

\bibitem[LW21]{LW21} M. Lipnowski and A. Wright, Towards optimal
spectral gap in large genus, (2021). Preprint, arXiv:2103.07496.

\bibitem[LM22]{LM22}L. Louder, M. Magee with Appendix by W. Hide,
and M. Magee. Strongly convergent unitary representations of limit
groups, Preprint, arXiv:2210.08953, (2022).

\bibitem[LRS95]{LRS95}W. Luo, Z. Rudnick and P. Sarnak. On Selberg\textquoteright s
eigenvalue conjecture, Geom. Funct. Anal. 5 (1995), no. 2, 387--401.

\bibitem[MN20]{MN20}M. Magee and F. Naud. Explicit spectral gaps
for random covers of Riemann surfaces. Publ. Math. Inst. Hautes Etudes
Sci (2020).

\bibitem[MN21]{MN21}M. Magee and F. Naud. Extension of Alon\textquoteright s
and Friedman\textquoteright s conjectures to Schottky surfaces, preprint,
arXiv:2106.02555. (2021)

\bibitem[MNP22]{MNP22}M. Magee, F Naud and D. Puder. A random cover
of a compact hyperbolic surface has relative spectral gap $\frac{3}{16}-\varepsilon$
. Geom. Funct. Anal. 32, 595--661 (2022). 

\bibitem[Mi13]{Mi13} M. Mirzakhani. Growth of Weil-Petersson volumes
and random hyperbolic surfaces of large genus, J. Differential. Geom.
94 (2013), no. 2, 267--300.

\bibitem[Mo21]{Mo21}L. Monk, Benjamini-Schramm convergence and spectrum
of random hyperbolic surfaces of high genus, Anal. PDE (In Press)
(2021).

\bibitem[Na22]{Na22}F. Naud, Random covers of compact surfaces and
smooth linear spectral statistics, arXiv:2209.07941 (2022)

\bibitem[Pi96]{Pi96}G. Pisier, A simple proof of a theorem of Kirchberg
and related results on $C^{*}$-norms, J. Operator Theory 35, 2 (1996),
pp. 317-{}-335.

\bibitem[RS81]{RS81}M. Reed, B. Simon. Methods of Modern Mathematical
Physics: Functional analysis. I. Academic Press, New York, (1981).

\bibitem[Ru22]{Ru22}Z. Rudnick, GOE statistics on the moduli space
of surfaces of large genus, arXiv:2202.06379 (2022).

\bibitem[Sa95]{Sa95}P. Sarnak. Selberg\textquoteright s eigenvalue
conjecture, Notices Amer. Math. Soc. 42 (1995), no. 11, 1272--1277.

\bibitem[Sa03]{Sa03}P. Sarnak. Spectra of Hyperbolic Surfaces, Bull.
Amer. Math. Soc. Vol. 40, No. 4, (2003), 441-478.

\bibitem[Se65]{Se65}A. Selberg, On the estimation of Fourier coefficients
of modular forms, Proc. Sympos. Pure Math., Vol. VIII, Amer. Math.
Soc., Providence, R.I., (1965) pp. 1--15.

\bibitem[SW22A]{SW22A}Y. Shen, Y. Wu, The Cheeger Constants of Random
Belyi Surfaces, International Mathematics Research Notices, (2022).

\bibitem[SW22]{SW22}Y. Shen, Y. Wu. Arbitrarily small spectral gaps
for random hyperbolic surfaces with many cusps,Preprint , arXiv:2203.15681,
(2022).

\bibitem[WX21]{WX21} Wu, Y., Xue, Y. Random hyperbolic surfaces of
large genus have first eigenvalues greater than $\frac{3}{16}-\epsilon$
. Geom. Funct. Anal. 32, 340--410 (2022).

\bibitem[Za22]{Za22}M. Zargar, Random flat bundles and equidistribution,
Preprint, arXiv:2210.09547 (2022).

\bibitem[Zo87]{Zo87}P. Zograf, Small eigenvalues of automorphic Laplacians
in spaces of parabolic forms. J Math Sci 36, 106--114 (1987).\linebreak{}

\end{thebibliography}
\end{document}